\documentclass[onefignum,onetabnum]{siamart171218}
\usepackage{amsmath,amsfonts,amssymb}
\usepackage{mathtools}
\usepackage{graphicx}
\ifpdf
  \DeclareGraphicsExtensions{.pdf,.png,.jpg}
\else
  \DeclareGraphicsExtensions{.eps}
\fi

\usepackage{algpseudocode} % 'algorithm' is loaded by siamart171218
\usepackage{booktabs}
\usepackage[caption=false]{subfig}
\usepackage{xcolor}
\usepackage{comment}

\colorlet{siaminlinkcolor}{black}
\colorlet{siamexlinkcolor}{black}
\AtBeginDocument{\hypersetup{
  linkcolor=black,
  citecolor=black,
  urlcolor=black,
  filecolor=black,
  runcolor=black,
  menucolor=black,
  pdfborder={0 0 0}
}}

\newif\ifshowcomments
\showcommentsfalse % set true to display internal notes

\ifshowcomments

\newcommand{\todo}[1]{{\color{red}\textbf{【要追記】} #1}}
\newenvironment{proposed}{\par\begingroup\color{blue}\noindent\textbf{【青字メモ（私案）】}\par}{\par\endgroup}
\else

\newcommand{\todo}[1]{}
\excludecomment{proposed}
\fi

\headers{Iterative Refinement for a Subset of Eigenvectors}{T. Terao et al.}

\title{Iterative Refinement for a Subset of Eigenvectors of Symmetric Matrices via Matrix Multiplications}
\author{Takeshi Terao\thanks{Research Institute for Science and Engineering, Waseda University, 3-4-1 Okubo, Shinjuku-ku, Tokyo 169-8555, Japan (\email{takeshi.terao@aoni.waseda.jp}).}
\and Katsuhisa Ozaki\thanks{Department of Mathematical Sciences, Shibaura Institute of Technology, 307 Fukasaku, Minuma-ku, Saitama, 337-8570, Saitama, Japan. }
\and Toshiyuki Imamura\thanks{RIKEN Center for Computational Science, 7-1-26 Minatojima-minami-machi, Chuo-ku, Kobe, Hyogo 650-0047, Japan. }
\and Takeshi Ogita\thanks{Department of Applied Mathematics, Faculty of Science and Engineering, Waseda University, 3-4-1 Okubo, Shinjuku-ku, Tokyo 169-8555, Japan }
}

\begin{document}

\maketitle

\begin{abstract}
We develop an iterative refinement method that improves the accuracy of a user-chosen subset of $k$ eigenvectors ($k\ll n$) of an $n\times n$ real symmetric matrix.
Using an orthogonal matrix represented in compact WY form, the method expresses the eigenvector error through a correction matrix that can be approximated efficiently from Rayleigh quotients and residuals.
Unlike refinement methods for a single eigenpair or for a full eigenbasis, the proposed method refines only the selected $k$ eigenvectors using $\mathcal{O}(nk)$ additional storage, and its dominant work can be organized as matrix--matrix multiplications.
Under an eigenvalue separation condition, the refinement converges linearly; we also provide a conservative sufficient condition.
Practical variants of the separation condition (e.g., via shifting) enable targeting other extremal parts of the spectrum.
For tightly clustered eigenvalues, we discuss limitations and show that preprocessing can restore convergence in a representative sparse example.
Numerical experiments on dense test matrices and sparse matrices from the SuiteSparse Matrix Collection illustrate attainable accuracy and problem-dependent convergence.
\end{abstract}



\section{Introduction}\label{sec:introduction}

The eigenvalue problem for real symmetric matrices plays a central role in a wide range of scientific and engineering applications. Given a real symmetric matrix $A \in \mathbb{R}^{n \times n}$, one seeks eigenpairs
\[
A x^{(i)} = \lambda_i x^{(i)}, \qquad i = 1,\dots,n,
\]
which are essential in electronic structure calculations, structural analysis, machine learning, graph computations, and many other fields. Modern numerical linear algebra libraries such as LAPACK and ARPACK allow these eigenpairs to be computed in a backward stable manner~\cite{anderson1999lapack,lehoucq1998arpack}. However, classical perturbation theory shows that when eigenvalues are clustered, the accuracy of the corresponding eigenvectors may deteriorate significantly. In practice, it is often necessary to apply techniques that improve the accuracy of the eigenvectors.

To enhance the accuracy of computed eigenpairs, Newton-type iterative refinement methods have been studied for decades. Specifically, the methods of Dongarra--Moler--Wilkinson and Tisseur improve the accuracy of a single eigenpair~\cite{dongarra1983improving,tisseur2001newton}. 
These approaches update one eigenpair at a time by solving a correction equation involving $A-\mu I$. In the dense case, each refinement step can cost $\mathcal{O}(n^3)$, and the correction equation may become ill-conditioned when the target eigenvalue is close to others; consequently, convergence can stagnate or become very slow for clustered eigenvalues. 
See~\cite{tsai2022mixed} for related work on mixed-precision refinement in a similar spirit.

More recently, Ogita and Aishima proposed refinement methods that simultaneously improve all eigenpairs and analyzed their efficiency and convergence~\cite{ogita2018iterative,ogita2019iterative}. 
See also~\cite{shiroma2019tracking,uchino2024high,uchino2022acceleration} for related developments.
Because their framework refines a full eigenbasis, it requires storing and updating all $n$ eigenvectors, which can be impractical in both runtime and memory when only a small subset is needed.
In contrast, our focus is on selectively improving only $k\ll n$ eigenvectors of interest.

More broadly, mixed-precision numerical algorithms have attracted significant attention, with iterative refinement being a canonical example. For surveys, see~\cite{abdelfattah2021survey,higham2022mixed}.

In many practical applications, one is interested not in the full eigenbasis but only in $k \ll n$ eigenvectors.
Classical refinement techniques are typically formulated either for improving a single eigenpair (via a correction equation) or for refining a full eigenbasis; comparatively less attention has been devoted to refinement methods that target only a selected subset while using $\mathcal{O}(nk)$ storage and organizing the dominant work as matrix--matrix multiplications.

In this paper, we propose a new iterative refinement method for a subset of eigenvectors $X = (x^{(p_1)}, \dots, x^{(p_k)}) \in \mathbb{R}^{n \times k}$, where $p$ is a permutation of $\{1,\dots,n\}$ and $p_1,\dots,p_k$ index the targeted eigenvectors.
The proposed framework constructs an orthogonal matrix $H \in \mathbb{R}^{n \times n}$ such that
\[
H^\mathsf{T} X =
\begin{pmatrix}
I_k \\O_{n-k,k}
\end{pmatrix},
\qquad
H^\mathsf{T} A X =
\begin{pmatrix}
D_1 \\ O_{n-k,k}
\end{pmatrix},
	\]
	where $D_1 = \mathrm{diag}(\lambda_{p_1},\dots,\lambda_{p_k})$ and $O_{n-k,k}$ denotes the $(n-k)\times k$ zero block.
Under this representation, the exact and approximate eigenvectors satisfy
\[
X = \widehat{X} + \widehat{H} E_L,
\]
for some (unknown) correction matrix $E_L \in \mathbb{R}^{n \times k}$.
The refinement updates
\[
\widehat{X} \leftarrow \widehat{X} + \widehat{H} \widetilde{E}_L
\]
where $\widetilde{E}_L$ is a computable approximation of $E_L$.

We first present an efficient construction of the orthogonal matrix $H$ using the compact WY representation of Householder transformations and a modified LU factorization, requiring only $\mathcal{O}(nk)$ storage~\cite{ballard2015reconstructing,schreiber1989storage,yamamoto2012aggregation}.
We then derive an algorithm for computing $\widetilde{E}_L$ from Rayleigh quotients and residuals.
Finally, we provide an error-block analysis and show that linear convergence is obtained under the eigenvalue separation condition
\begin{align}
\max_{k+1\le j\le n} |\lambda_{p_j}| < \min_{1\le i\le k} |\lambda_{p_i}|.\label{eq:condition}
\end{align}
If we choose $p$ so that $|\lambda_{p_1}|\ge \cdots \ge |\lambda_{p_n}|$, then \eqref{eq:condition} reduces to $|\lambda_{p_{k+1}}|<|\lambda_{p_k}|$; equivalently, $\min_{1\le i\le k} |\lambda_{p_i}|=|\lambda_{p_k}|$ and $\max_{k+1\le j\le n} |\lambda_{p_j}|=|\lambda_{p_{k+1}}|$.
In particular, the proposed refinement is effective for eigenvectors corresponding to eigenvalues of larger magnitude, while the convergence deteriorates as the target eigenvalues become more tightly clustered. These behaviors are theoretically justified within our framework.

This eigenvalue separation condition implies that the target eigenvectors cannot be chosen completely arbitrarily.
However, since the condition is stated in terms of the ordered eigenvalues $\{\lambda_{p_i}\}$, it admits practically useful variants through different choices of the ordering and through simple transformations such as shifting (or shift-and-square) that make the desired eigenvectors dominant in a transformed problem; see Section~\ref{sec:variants}.
Nevertheless, it covers many application-relevant cases where the eigenvectors associated with dominant modes must be computed with high accuracy.

Moreover, the proposed method is not a subspace projection method.
In particular, it does not compute Ritz pairs via Rayleigh--Ritz on a fixed trial subspace.
Thus, it does not converge to a projected solution for that subspace.
Instead, it aims to converge to eigenvectors of the original matrix~$A$.

Finally, the proposed method aligns well with current hardware trends toward mixed precision and matrix-multiplication-accelerated hardware~\cite{11196413}.
The workflow separates an inexpensive low-precision initial approximation from a refinement phase whose dominant operations can be implemented using matrix--matrix multiplications (e.g., products with $A$ and with $I-YTY^\mathsf{T}$), together with $\mathcal{O}(k^3)$ operations on small $k\times k$ dense matrices (e.g., forming and applying $\widehat T$).
As a result, the refinement is amenable to high-performance implementations based on optimized level-3 BLAS kernels.

\begin{proposed}
The mixed-precision / BLAS3 message is already stated in the main text, so avoid adding another paragraph with essentially the same claim.
Instead, add only a few representative references (2--4) on mixed-precision eigensolvers / matrix-multiplication-oriented eigensolvers / accelerator-optimized symmetric eigensolvers, and cite them at the end of the paragraph above.
\end{proposed}

While the refinement phase incurs additional cost beyond the initial eigensolver, we emphasize that the objective of the proposed method is not necessarily to outperform a fully optimized double-precision eigensolver in wall-clock time.
Rather, it provides a mechanism for recovering high accuracy for a selected subset of eigenvectors starting from a low-precision or otherwise low-cost initial approximation.
A comprehensive performance study against tuned double-precision solvers and optimized implementations on modern accelerators is an important topic for future work.

We demonstrate the effectiveness of the proposed method through numerical experiments on dense test problems and sparse matrices from the SuiteSparse Matrix Collection~\cite{davis2011uf}.
For most matrices, highly accurate eigenvectors are obtained, with convergence behavior consistent with the theory.
When the eigenvalue gap is extremely small, the plain refinement iteration can stagnate or diverge; however, suitable preprocessing can restore convergence in a representative sparse example.
These results provide numerical evidence regarding both the stability and the limitations of the proposed method.

This paper substantially extends our preliminary JSIAM Letters note~\cite{terao2024iterative}, which introduced the core algorithmic framework.
In the present paper, we provide a convergence analysis, validate and illustrate the analysis through numerical experiments, and discuss practical variants and extensions that broaden the range of eigenvalue problems to which the refinement can be applied.

The contributions of this paper are as follows:
\begin{itemize}
\item We propose an iterative refinement framework that improves the accuracy of a selected subset of $k\ll n$ eigenvectors of a real symmetric matrix, requiring only $\mathcal{O}(nk)$ storage and enabling implementations whose dominant kernels are matrix--matrix multiplications.
\item We derive an efficient refinement step that approximates the leading correction matrix from Rayleigh quotients and residuals, and we show how to apply the underlying orthogonal matrices via compact WY factors without forming them explicitly.
\item We provide an error-block analysis and a convergence result that yields a linear convergence factor under an eigenvalue separation condition, and we clarify how convergence degrades for clustered eigenvalues.
\item We report numerical experiments on dense test problems and sparse matrices from the SuiteSparse Matrix Collection, illustrating both successful and failure regimes and demonstrating how preprocessing can restore convergence in a representative clustered case.
\end{itemize}

The remainder of this paper is organized as follows. Section~\ref{sec:notation} introduces notation. Section~\ref{sec:main-results} presents the proposed refinement method and its analysis. Section~\ref{sec:numerical-experiments} reports numerical experiments, Section~\ref{sec:variants} discusses several variants of the proposed method, and Section~\ref{sec:conclusion} concludes the paper.

\todo{(For a SIMAX submission) If switching to the official \texttt{siamart} template, add the required items (e.g., Keywords/AMS(MSC), affiliation/acknowledgments). For a clean submission, delete/disable all \texttt{\\todo\{...\}}.}

\section{Notation}\label{sec:notation}

For a matrix $A = (a_{ij}) \in \mathbb{R}^{m \times n}$, the entry in the $i$th row and $j$th column is denoted by $a_{ij}$, and the $j$th column of $A$ is denoted by $a_j \in \mathbb{R}^{m}$.  
For integers $1 \le i_1 \le i_2 \le m$ and $1 \le j_1 \le j_2 \le n$, we use
\[
A(i_1\!:\!i_2,\; j_1\!:\!j_2)
\]
to denote the block submatrix of $A$ obtained by extracting rows $i_1,\dots,i_2$ and columns $j_1,\dots,j_2$.

The notation $\|\cdot\|$ refers to the spectral norm, and $\|\cdot\|_F$ denotes the Frobenius norm.  

The matrix $I_k$ denotes the $k \times k$ identity matrix.  
The matrix $I_{n,k}$ denotes the first $k$ columns of $I_n$.
The matrix $O_{m,n} \in \mathbb{R}^{m \times n}$ denotes the zero matrix, and we write $O_n$ when $m=n$.

\section{Main Results}\label{sec:main-results}

Let $(\lambda_i, x^{(i)}) \in \mathbb{R} \times \mathbb{R}^n$ denote the eigenpairs of a real symmetric matrix $A$, where the eigenvectors are normalized as $\|x^{(i)}\| = 1$ for $i = 1, \dots, n$.  
We focus on a subset of $k < n$ eigenvectors and define
\[
X = (x^{(p_1)}, x^{(p_2)}, \dots, x^{(p_k)}) \in \mathbb{R}^{n \times k}.
\]

We assume the existence of an orthogonal matrix $H \in \mathbb{R}^{n \times n}$ satisfying
\begin{align}
    H^\mathsf{T} X &=
    \begin{pmatrix}
        I_k \\
        O_{n-k,k}
    \end{pmatrix}, \qquad
    H^\mathsf{T} A X =
    \begin{pmatrix}
        D_1 \\
        O_{n-k,k}
    \end{pmatrix}, \label{eq:HX-and-HAX}
\end{align}
where
\[
D_1 = \mathrm{diag}(\lambda_{p_1}, \dots, \lambda_{p_k}).
\]

Furthermore, we assume
\begin{align}
    H^\mathsf{T} H &= I_n, \qquad
    H^\mathsf{T} A H =
    \begin{pmatrix}
        D_1 & O_{n-k,k}^\mathsf{T} \\
        O_{n-k,k} & D_2
    \end{pmatrix},
    \label{eq:HTH-and-HTAH}
\end{align}
where $D_2 \in \mathbb{R}^{(n-k)\times(n-k)}$ is symmetric.

Let $\widehat{X}$ and $\widehat{H}$ denote approximations of $X$ and $H$.  
We introduce a correction matrix $E \in \mathbb{R}^{n \times n}$ by
\begin{align}
    H = \widehat{H} + \widehat{H} E.
    \label{eq:H-E-def}
\end{align}

Using the block partition
\begin{align}
    E =
    \begin{pmatrix}
        E_{11} & E_{12} \\
        E_{21} & E_{22}
    \end{pmatrix},
    \quad
    E_L =
    \begin{pmatrix}
        E_{11} \\
        E_{21}
    \end{pmatrix}, \quad
    E_{11}\in\mathbb{R}^{k\times k}, \ \ E_{21},E_{12}^\mathsf{T}\in\mathbb{R}^{(n-k)\times k},
    \label{eq:E-block}
\end{align}
we seek an approximation $\widetilde{E}_L$ of $E_L$.

The refinement of the approximate eigenvector matrix $\widehat{X}$ is then carried out by iterating
\begin{align}
    \widehat{X} \leftarrow \widehat{X} + \widehat{H} \, \widetilde{E}_L,
    \label{eq:update-X}
\end{align}
which progressively improves the accuracy of $\widehat{X}$.

\subsection{Algorithm}

\subsubsection{Construction of the orthogonal matrix \texorpdfstring{$H$}{H}}

A memory-efficient way to represent and apply such orthogonal matrices is the compact WY representation~\cite{schreiber1989storage}.
We describe how to construct an orthogonal matrix $H \in \mathbb{R}^{n \times n}$ such that $H^\mathsf{T}X=I_{n,k}$
from a matrix $X \in \mathbb{R}^{n \times k}$, following the approaches in~\cite{ballard2015reconstructing,yamamoto2012aggregation}.

We assume that the columns of $X$ are mutually orthogonal and that $k < n$.
In the compact WY representation, an orthogonal matrix $H$ is written as
\[
H = I_n - Y T Y^\mathsf{T},
\qquad
Y \in \mathbb{R}^{n \times k}, \quad
T \in \mathbb{R}^{k \times k}.
\]
A key advantage of this representation is that $H$ can be stored using only $\mathcal{O}(nk)$ memory,
and the full $n \times n$ matrix need not be formed explicitly.

By choosing $Y$ and $T$ such that
\[
H^\mathsf{T} X = I_{n,k},
\]
the orthogonal matrix $H$ maps the column space of $X$ onto the first $k$ coordinate directions.
Here, $I_{n,k}$ denotes the first $k$ columns of $I_n$.

Yamamoto~\cite{yamamoto2012aggregation} proposed Algorithm~\ref{alg:compwy},
which efficiently constructs $Y$ and $T$.
In particular, when the columns of $X$ are orthogonal, one can set $Y = Q - I_{n,k}$,
leading to a simple construction of~$H$.

\begin{algorithm}
  \caption{$[Y,T]=\mathrm{CompactWY}(Q)$~\cite{yamamoto2012aggregation}}\label{alg:compwy}
  \begin{algorithmic}[1]
    \Require $Q\in\mathbb{R}^{n\times k}$ with orthonormal columns
    \State $Y = Q - I_{n,k}$
    \State $T = -Y_1^{-\mathsf{T}}$ \hfill\% $Y_1$ is the leading $k \times k$ block of $Y$
    \Ensure $Q = (I - Y T Y^\mathsf{T}) I_{n,k}$ and $I-YTY^\mathsf{T}$ is orthogonal
  \end{algorithmic}
\end{algorithm}

The crucial requirement for Algorithm~\ref{alg:compwy} is that the leading block
$Y_1 \in \mathbb{R}^{k \times k}$ is nonsingular.
In floating-point arithmetic, however, $Y_1$ may be singular or nearly singular,
and computing $Y_1^{-1}$ may become unstable.

To overcome this issue, Ballard~\cite{ballard2015reconstructing}
introduced a modified LU factorization (Algorithm~\ref{modifiedlu}),
which enables the construction of $H$ even when $Y_1$ is singular.
This method performs an LU factorization of $Q$ while applying sign corrections,
yielding a triangular factorization $Q - \Sigma = YU$,
where $\Sigma$ is a diagonal matrix determined by the signs of the diagonal entries of $Q$.

Although robust, the modified LU method requires $\mathcal{O}(nk^2)$ computational cost,
which is significantly higher than that of the compact WY method.

\begin{algorithm}
  \caption{$[Y,U,\Sigma]=\mathrm{modified\_LU}(Q)$}\label{modifiedlu}
  \begin{algorithmic}[1]
    \Require $Q\in\mathbb{R}^{n\times k}$
    \For{$j=1,\dots,k$}
        \State $\sigma_j = -\mathrm{sign}(q_{jj})$
        \State $q_{jj} = q_{jj} - \sigma_j$
        \State $q_{lj} = q_{lj} / q_{jj}$ \quad ($j+1 \le l \le n$)
        \State $q_{lh} = q_{lh} - q_{lj} q_{jh}$ \quad ($j+1 \le l \le n$, $j+1 \le h \le k$)
    \EndFor
    \State $\Sigma = \mathrm{diag}(\sigma)$
    \Ensure The upper triangular part stores $U$, the lower triangular part stores $Y$, satisfying $Q-\Sigma = YU$
  \end{algorithmic}
\end{algorithm}

Using the modified LU factorization, $Y_1$ need not be nonsingular.
Based on Ballard's result, Algorithm~\ref{alg:compwy_lu} constructs $H$
for any $Q$ with orthonormal columns.

\begin{algorithm}
  \caption{$[Y,T,\Sigma]=\mathrm{CompactWY\_LU}(Q)$~\cite{ballard2015reconstructing}}\label{alg:compwy_lu}
  \begin{algorithmic}[1]
    \Require $Q\in\mathbb{R}^{n\times k}$ with orthonormal columns
    \State $[Y,U,\Sigma] = \mathrm{modified\_LU}(Q)$ \hfill\% Algorithm~\ref{modifiedlu}
    \State $T = -U \Sigma Y_1^{-\mathsf{T}}$ \hfill\% $Y_1$ is the leading $k \times k$ block of $Y$
    \Ensure $Q = (I - Y T Y^\mathsf{T}) I_{n,k}$ and $I-YTY^\mathsf{T}$ is orthogonal
  \end{algorithmic}
\end{algorithm}

The compact WY method is fast and memory-efficient but depends on the nonsingularity of $Y_1$.
The modified LU method removes this restriction and works in more general cases,
at the cost of higher computational complexity.
In practice, Algorithm~\ref{alg:compwy} is often applicable and preferable.
Even when Algorithm~\ref{alg:compwy} fails,
one can apply a one-time transformation $\widehat{X} \leftarrow \widehat{X} \Sigma$,
after which Algorithm~\ref{alg:compwy} succeeds for subsequent iterations.

\subsubsection{Iterative refinement for a subset of eigenvectors}

We now describe how to compute an approximation $\widetilde{E}_L$ of the leading correction block $E_L$.
Assume that $\|E\| < 1$.
Then $I + E$ is invertible and admits the Neumann expansion
\begin{align}
    (I_n + E)^{-1}
    &= I_n - E + \Delta_E,
    \qquad
	    \Delta_E = \sum_{\ell=2}^{\infty} (-E)^\ell,
    \qquad
    \|\Delta_E\| \le \frac{\|E\|^2}{1-\|E\|}.
    \label{eq:Neumann}
\end{align}

Let $\Phi_1$ and $\Phi_2$ satisfy
\begin{align}
    \widehat{H}^\mathsf{T}\widehat{H}
    &= (I - E + \Delta_E)^\mathsf{T}(I-E+\Delta_E)
       = I_n - E - E^\mathsf{T} + \Phi_1,
      \label{eq:HhatHhat}\\
    \widehat{H}^\mathsf{T}A\widehat{H}
    &= (I - E + \Delta_E)^\mathsf{T} D (I - E + \Delta_E)
       = D - D E - E^\mathsf{T} D + \Phi_2,
      \label{eq:HhatA}
\end{align}
where $D = \mathrm{diag}(D_1, D_2)$.

Using \eqref{eq:E-block} and \eqref{eq:HhatHhat}, we obtain
\begin{align}
    \widehat{H}^\mathsf{T}\widehat{X}
    =
    \begin{pmatrix}
        I_k - E_{11} - E_{11}^\mathsf{T}\\[1mm]
        -E_{21} - E_{12}^\mathsf{T}
    \end{pmatrix}
    + \Delta_1,
    \qquad
    \Delta_1 = \Phi_1(1\!:\!n,\,1\!:\!k).
    \label{eq:HhatX-block}
\end{align}
Similarly, from \eqref{eq:HhatA} we obtain
\begin{align}
    \widehat{H}^\mathsf{T} A \widehat{X}
    &=
    \begin{pmatrix}
        D_1 - D_1 E_{11} - E_{11}^\mathsf{T} D_1\\[1mm]
        -E_{12}^\mathsf{T} D_1
    \end{pmatrix}
    + \Delta_2,
    \ \ 
    \Delta_2 =
    \Phi_2(1\!:\!n,1\!:\!k)
    -
    \begin{pmatrix}
        O_k\\
        D_2 E_{21}
    \end{pmatrix}.
    \label{eq:HhatAX-block}
\end{align}

Neglecting $\Delta_1$ and $\Delta_2$, we obtain the approximations
\begin{align}
    \widehat{H}^\mathsf{T}\widehat{X}
    &=
    \begin{pmatrix}
        I_k - \widetilde{E}_{11} - \widetilde{E}_{11}^\mathsf{T}\\[1mm]
        -\widetilde{E}_{21} - \widetilde{E}_{12}^\mathsf{T}
    \end{pmatrix},
    \label{eq:R}\\
    \widehat{H}^\mathsf{T}A\widehat{X}
    &=
    \begin{pmatrix}
        \widetilde{D}_1 - \widetilde{D}_1 \widetilde{E}_{11} - \widetilde{E}_{11}^\mathsf{T}\widetilde{D}_1\\[1mm]
        -\widetilde{E}_{12}^\mathsf{T}\widetilde{D}_1
    \end{pmatrix},
    \label{eq:S}
\end{align}

Following \cite{ogita2018iterative}, the diagonal entries of $\widetilde{D}_1$ are computed by the Rayleigh quotient:
\begin{align}
    \widetilde{d}_j
    =
    \frac{\widehat{x}_j^\mathsf{T} A \widehat{x}_j}{\widehat{x}_j^\mathsf{T} \widehat{x}_j},
    \qquad j=1,\dots,k.
    \label{eq:Rayleigh}
\end{align}
Similarly, the diagonal entries of $\widetilde{E}_L$ satisfy
\begin{align}
    \widetilde{e}_{jj}
    = \frac{1-\widehat{x}_j^\mathsf{T}\widehat{x}_j}{2},
    \qquad j=1,\dots,k.
    \label{eq:Ediag}
\end{align}

From \eqref{eq:R} and \eqref{eq:S}, define
\begin{align}
    V
    :={}\widehat{H}^\mathsf{T}(A\widehat{X} - \widehat{X}\widetilde{D}_1)
    =
    \begin{pmatrix}
        \widetilde{D}_1 \widetilde{E}_{11} - \widetilde{E}_{11} \widetilde{D}_1\\[1mm]
        \widetilde{E}_{21} \widetilde{D}_1
    \end{pmatrix}.
    \label{eq:V-def}
\end{align}
The off-diagonal entries of $\widetilde{E}_L$ are given by
\begin{align}
    \widetilde{e}_{ij}
    =
    \begin{cases}
        \displaystyle \frac{v_{ij}}{\widetilde{d}_j - \widetilde{d}_i},
        & 1\le i,j\le k,\ \widetilde{d}_i\ne\widetilde{d}_j,\\[6pt]
        \displaystyle \frac{v_{ij}}{\widetilde{d}_j},
        & k+1\le i\le n,
    \end{cases}
    \label{eq:Eoffdiag}
\end{align}

Based on \eqref{eq:Rayleigh}, \eqref{eq:Ediag}, and \eqref{eq:Eoffdiag},
Algorithm~\ref{alg:proposed} summarizes the computation of~$\widetilde{E}_L$.

\begin{algorithm}
\caption{Refinement algorithm for a subset of eigenvectors}\label{alg:proposed}
\begin{algorithmic}[1]
    \Require A real symmetric matrix $A\in\mathbb{R}^{n\times n}$ and an approximate eigenvector matrix $\widehat{X}\in\mathbb{R}^{n\times k}$
    \State $W = A\widehat{X}$
    \State Compute $\widehat{Y}$ and $\widehat{T}$ using Algorithm~\ref{alg:compwy} or Algorithm~\ref{alg:compwy_lu}
    \State $\alpha_j = \widehat{x}_j^\mathsf{T}\widehat{x}_j$ for $j=1,\dots,k$
    \State $\widetilde{d}_j = (\widehat{x}_j^\mathsf{T} w_j)/\alpha_j$ for $j=1,\dots,k$
    \State $V = \widehat{H}^\mathsf{T}(W - \widehat{X}\widetilde{D}_1)$
    \State $
    \widetilde{e}_{ij} =
    \begin{cases}
        (1-\alpha_i)/2, & i=j,\\
        \beta_{ij}, & 1 \le i,j \le k,\ |\widetilde{d}_j-\widetilde{d}_i|\le\delta,\\
        v_{ij}/(\widetilde{d}_j-\widetilde{d}_i), & 1 \le i,j \le k,\ |\widetilde{d}_j-\widetilde{d}_i|>\delta,\\
        v_{ij}/\widetilde{d}_j, & k+1 \le i \le n,
    \end{cases}
    $
    \State \Return $\widetilde{X} = \widehat{X} + \widehat{H}\widetilde{E}_L$
    \Ensure An updated approximate eigenvector matrix $\widetilde{X}\in\mathbb{R}^{n\times k}$
\end{algorithmic}
\end{algorithm}

In Algorithm~\ref{alg:proposed}, $\delta>0$ is a threshold used to detect (near-)clustered Rayleigh quotients and avoid division by small denominators.
The choice of $\beta_{ij}$ in Algorithm~\ref{alg:proposed} depends on the strategy used to handle (near-)clustered Rayleigh quotients.
For clustered eigenvalues,
\[
\beta_{ij} = -\widehat{x}_i^\mathsf{T}\widehat{x}_j/2
\]
is used~\cite{ogita2019iterative}.
When a preprocessing step is applied, $\beta_{ij}=0$ is used instead~\cite{shiroma2019tracking}.

Table~\ref{tab:cost} summarizes the computational cost of one refinement iteration in Algorithm~\ref{alg:proposed}.
\begin{table}[htbp]
\centering
\caption{Cost per iteration of Algorithm~\ref{alg:proposed}.}
\label{tab:cost}
\begin{tabular}{rll}
\toprule
Line(s) & Operation & Cost\\
\midrule
1 & Multiply by $A$ ($W=A\widehat X$) & $\mathcal{O}(\mathrm{nnz}(A)k)$ flops\\
2 & Build compact WY factors ($\widehat Y,\widehat T$) & $\mathcal{O}(k^3)$ flops\\
5,7 & Apply $\widehat H^\mathsf{T}$ and $\widehat H$ & $4nk^2+2k^3$ flops\\
others & Vector ops and small dense ops & $\mathcal{O}(nk+k^2)$ flops\\
\bottomrule
\end{tabular}
\end{table}

Here, $\mathrm{nnz}(A)$ denotes the number of nonzero entries of $A$ (so $\mathrm{nnz}(A)=n^2$ in the dense case).
The applications of $\widehat H$ and $\widehat H^\mathsf{T}$ in lines 5 and 7 can be implemented without forming $\widehat H$ explicitly:
using the compact WY representation $\widehat H=I_n-\widehat Y\widehat T\widehat Y^\mathsf{T}$ with $\widehat Y\in\mathbb{R}^{n\times k}$ and $\widehat T\in\mathbb{R}^{k\times k}$,
we compute for an $n\times k$ matrix $B$
\begin{align}
    \widehat H B = B - \widehat Y\bigl(\widehat T(\widehat Y^\mathsf{T}B)\bigr),
\end{align}
and similarly for $\widehat H^\mathsf{T}B$, which requires only $\mathcal{O}(nk)$ storage.

\subsubsection{Implementation details}

We describe practical implementation choices for Algorithm~\ref{alg:proposed}, focusing on stopping criteria, precision, and the parameter~$\delta$.

\paragraph{Stopping criteria}
The refinement can be terminated using either a backward-error criterion or a correction-size criterion.
A backward-error criterion monitors the relative eigenpair residual
\[
\frac{\|A\widehat X-\widehat X\widetilde D_1\|_F}{\|A\|},
\]
and stops when it falls below a prescribed tolerance.
Alternatively, one may stop when the correction size $\|\widetilde E_L\|_F$ becomes smaller than a tolerance.
Since $X-\widehat X=\widehat H E_L$ and $E_L\approx \widetilde E_L$, this quantity serves as a practical proxy for the eigenvector error.
Both criteria require only $\mathcal{O}(nk)$ work and have negligible overhead compared with the dominant matrix--matrix kernels.

\paragraph{Precision and orthogonality}
If the goal is to recover double-precision-level accuracy, a simple choice is to perform all refinement computations in double precision, as done in the experiments.
Mixed-precision variants are also possible; for example, some matrix--matrix multiplications in the applications of $\widehat H$ and $\widehat H^\mathsf{T}$ may be computed in lower precision while keeping the accumulation and the small $k\times k$ computations in double.
A systematic study of such variants is left for future work.
To satisfy the assumptions of the compact WY construction, we normalize the columns of $\widehat X$ after each update and optionally reorthogonalize $\widehat X$ when loss of orthogonality becomes non-negligible.

\paragraph{Choice of $\delta$}
The parameter $\delta$ in Algorithm~\ref{alg:proposed} is used to detect (near-)clustered Rayleigh quotients and avoid division by small denominators.
In our MATLAB implementation, we use the heuristic
\[
\delta = c\,\|A\|\,u,
\]
where $u$ is the unit roundoff of the working precision and $c>1$ (we used $c=10$).

\subsection{Convergence Analysis}\label{sec:convergence-analysis}

We analyze the convergence conditions of Algorithm~\ref{alg:proposed}. Specifically, for the update
$\widetilde X= \widehat X+\widehat H\widetilde E_L$, we study conditions under which
\begin{align}
    \frac{\|X-\widetilde X\|}{\|X-\widehat X\|}<1
\end{align}
holds.
In particular, for each eigenvector $x^{(p_j)}$, $1\leq j\leq k$, we show linear convergence with the
asymptotic factor
\begin{align}
    \limsup_{\epsilon\to0}\frac{\|x^{(p_j)}-\widetilde x^{(p_j)}\|}{\|x^{(p_j)}-\widehat x^{(p_j)}\|}
    =
    \frac{\max_{k+1\leq i\leq n}|\lambda_{p_i}|}{|\lambda_{p_j}|}.
\end{align}

\subsubsection{Previous work}

We briefly review part of the analysis by Ogita and Aishima~\cite{ogita2018iterative}.

We first introduce the notation. Given an orthogonal matrix $H$ and its approximation $\widehat H$,
define the correction matrix $E$ by $H=\widehat H(I+E)$.
We set
\begin{align}
    \epsilon&:=\|E\|<1,\\
    \chi(\epsilon)&:=\frac{3-2\epsilon}{(1-\epsilon)^2},\\
    \eta(\epsilon)&:=\frac{2(1+2\epsilon+\chi(\epsilon)\epsilon^2)\chi(\epsilon)}{(1-\epsilon)(1-2\epsilon-\chi(\epsilon)\epsilon^2)},\\
    \omega(\epsilon)&:=2\chi(\epsilon)+2\eta(\epsilon)\epsilon+\chi(\epsilon)\eta(\epsilon)\epsilon^2
\end{align}
and assume $\epsilon<1/100$, in which case
\begin{align}
    \chi(\epsilon)< 3.05,\quad
    \eta(\epsilon)< 7,\quad
    \omega(\epsilon)< 6.4
\end{align}
as noted in~\cite{ogita2018iterative}.

\begin{lemma}[Ogita--Aishima~\cite{ogita2018iterative}]\label{lem:ogita1}
    For \eqref{eq:HhatHhat} and \eqref{eq:HhatA}, if $\epsilon=\|E\|<1$, then
    \begin{align}
        \|\Phi_1\|&\leq \chi(\epsilon)\epsilon^2\label{eq:norm_phi1},\\
        \|\Phi_2\|&\leq \chi(\epsilon)\|A\|\epsilon^2\label{eq:norm_phi2}
    \end{align}
    hold.
\end{lemma}

\begin{lemma}[Ogita--Aishima~\cite{ogita2018iterative}]\label{lem:ogita2}
    Let $A$ be a real symmetric $n\times n$ matrix with simple eigenvalues $\lambda_i, i=1,\dots,n$, and let $p$ be a permutation of $\{1,\dots,n\}$.
    Define $X=(x^{(p_1)},x^{(p_2)},\dots,x^{(p_k)})\in\mathbb{R}^{n\times k}$ with $k<n$.
    For a given full column rank $\widehat X\in\mathbb{R}^{n\times k}$, suppose that Algorithm~\ref{alg:proposed} is applied to $A$ and $\widehat X$ in exact arithmetic.
    If 
    \begin{align}
        \epsilon\leq \sqrt{\frac{\min_{\substack{1\le i\le k\\ i\ne j}}|\lambda_{p_j}-\lambda_{p_i}|}{2\eta(\epsilon)\|A\|}}
    \end{align}
    then we obtain
    \begin{align}
        |\lambda_{p_i}-\widetilde d_i|&\leq \eta(\epsilon)\|A\|\epsilon^2,\\
        |e_{ij}-\widetilde e_{ij}|&<\frac{\omega(\epsilon)\|A\|\epsilon^2}{|\lambda_{p_j}-\lambda_{p_i}|-2\eta(\epsilon)\|A\|\epsilon^2}\label{eq:E11}
    \end{align}
    for $1\leq i,j\leq k$.
\end{lemma}

\subsubsection{Analysis of the proposed method}

We next present lemmas needed for the analysis of the proposed method.
\begin{lemma}\label{lem:proposed1}
    Let $H$ be an $n\times n$ orthogonal matrix and $\widehat H$ be an $n\times n$ nonsingular matrix.
    Define $E$ such that $H=\widehat H(I+E)$.
    If $\epsilon=\|E\|<1/2$, then it holds
    \begin{align}
        \kappa_2(\widehat H)\leq \frac{1+\epsilon}{1-\epsilon}<\frac{1}{1-2\epsilon}
    \end{align}
\end{lemma}
\begin{proof}
    From $H=\widehat H(I+E)$, we have $\widehat H = H(I+E)^{-1}$.
    Since $H$ is orthogonal, it follows that
    \begin{align}
        \|\widehat H\|
        &=\|(I+E)^{-1}\|\leq \frac{1}{1-\epsilon},\\
        \sigma_{\min}(\widehat H)
        &=\sigma_{\min}((I+E)^{-1})
        =\frac{1}{\sigma_{\max}(I+E)}
        \geq \frac{1}{\|I+E\|}
        \geq \frac{1}{1+\epsilon}.
    \end{align}
    Therefore, $\kappa_2(\widehat H)\leq (1+\epsilon)/(1-\epsilon)$.
\end{proof}

\begin{lemma}\label{lem:proposed2}
    Let $A$ be a real symmetric $n\times n$ matrix with simple eigenvalues $\lambda_i$ ($i=1,\dots,n$).
    Let $p$ be a permutation of $\{1,\dots,n\}$.
    Define $X=(x^{(p_j)})_{j=1}^k\in\mathbb{R}^{n\times k}$ and assume $k<n$.
    Let $\widehat X\in\mathbb{R}^{n\times k}$ have full column rank.
    Suppose that Algorithm~\ref{alg:proposed} is applied to $A$ and $\widehat X$ in exact arithmetic.
    If 
    \begin{align}
        \epsilon<\sqrt{\frac{|\lambda_{p_j}|}{\eta(\epsilon)\|A\|}}
    \end{align}
    then it holds
    \begin{align}
        \|E(k+1:n,j)-\widetilde E(k+1:n,j)\|&<\frac{\max_{i>k} |\lambda_{p_i}|\epsilon+\omega(\epsilon)\|A\|\epsilon^2}{|\lambda_{p_j}|-\eta(\epsilon)\|A\|\epsilon^2}.
    \end{align}
\end{lemma}

\begin{proof}
    From the definition of $\Delta_1$ in \eqref{eq:HhatX-block}, we have
    \begin{align}
	        \Delta_1=
	        \begin{pmatrix}
	            (E_{11}-\widetilde E_{11})+(E_{11}-\widetilde E_{11})^\mathsf{T}\\
	            (E_{21}-\widetilde E_{21})+(E_{12}-\widetilde E_{12})^\mathsf{T}
	        \end{pmatrix}
	    \end{align}
    Next, define $\widetilde \Delta_2$ by
    \begin{align}
        \widetilde \Delta_2:=\widehat H^\mathsf{T}A\widehat X-
        \begin{pmatrix}
            \widetilde D_1-\widetilde D_1E_{11}-E_{11}^\mathsf{T}\widetilde D_1\\
            -E_{12}^\mathsf{T}\widetilde D_1
        \end{pmatrix}
    \end{align}
    Then, by \eqref{eq:S},
    \begin{align}
        \widetilde\Delta_2=
        \begin{pmatrix}
            -\widetilde D_1(E_{11}-\widetilde E_{11})-(E_{11}-\widetilde E_{11})^\mathsf{T}\widetilde D_1\\
            -(E_{12}-\widetilde E_{12})\widetilde D_1
        \end{pmatrix}
    \end{align}
    and hence
    \begin{align}
        \Delta_1\widetilde D_1+\widetilde\Delta_2=
        \begin{pmatrix}
            \widetilde D_1(E_{11}-\widetilde E_{11})-(E_{11}-\widetilde E_{11})\widetilde D_1\\
            (E_{21}-\widetilde E_{21})\widetilde D_1
        \end{pmatrix}
    \end{align}
    holds.

    Next, from the definition of $\Delta_2$ in \eqref{eq:HhatAX-block},
    \begin{align}
        \widetilde \Delta_2-\Delta_2=
        \begin{pmatrix}
            (D_1-\widetilde D_1)-(D_1-\widetilde D_1)E_{11}-E_{11}^\mathsf{T}(D_1-\widetilde D_1)\\
            -E_{12}^\mathsf{T}(D_1-\widetilde D_1)
        \end{pmatrix}
    \end{align}
    holds. Therefore, by Lemma~\ref{lem:ogita2},
    \begin{align}
        \|\widetilde \Delta_2(k+1:n,1:k)-\Delta_2(k+1:n,1:k)\|&\leq \epsilon\|D_1-\widetilde D_1\|\leq \eta(\epsilon)\|A\|\epsilon^3
    \end{align}
    holds.
    
    Moreover, from \eqref{eq:HhatAX-block} and \eqref{eq:norm_phi2},
    \begin{align}
        \|\Delta_2(k+1:n,j)\|&\leq \|\Phi_2\|+\|D_2\|\|E_{21}\|\leq \chi(\epsilon)\|A\|\epsilon^2+\max_{k+1\leq i\leq n}|\lambda_{p_i}|\epsilon
    \end{align}
    and hence
    \begin{align}
        \|\widetilde \Delta_2(k+1:n,j)\|&\leq (\eta(\epsilon)\epsilon+\chi(\epsilon))\|A\|\epsilon^2+\max_{k+1\leq i\leq n}|\lambda_{p_i}|\epsilon
    \end{align}
    holds.

    Combining these bounds, we obtain
    \begin{align}
        &\|E(k+1:n,j)-\widetilde E(k+1:n,j)\|\\
        \leq\quad& \frac{\| \Delta_1\|\cdot\|\widetilde D_1\|+\|\widetilde \Delta_2(k+1:n,1:k)\|}{|\widetilde d_j|}\\
        \leq\quad& 
        \frac{
        \begin{aligned}
        \chi(\epsilon)\epsilon^2\cdot\max_i\widetilde d_i
        +(\eta(\epsilon)\epsilon+\chi(\epsilon))\|A\|\epsilon^2
        +\max_{i>k} |\lambda_{p_i}|\epsilon
        \end{aligned}
        }{\min_i|\widetilde d_i|}\\
        \leq\quad& \frac{\max_{k+1\leq i\leq n} |\lambda_{p_i}|\epsilon+\omega(\epsilon)\|A\|\epsilon^2}{|\lambda_{p_j}|-\eta(\epsilon)\|A\|\epsilon^2}\label{eq:E21}
    \end{align}
    as desired.

\end{proof}

\begin{corollary}\label{cor:proposed2}
    Let $A$ be a real symmetric $n\times n$ matrix with simple eigenvalues $\lambda_i$ ($i=1,\dots,n$).
    Let $p$ be a permutation of $\{1,\dots,n\}$.
    Define $X=(x^{(p_j)})_{j=1}^k\in\mathbb{R}^{n\times k}$ and assume $k<n$.
    Let $\widehat X\in\mathbb{R}^{n\times k}$ have full column rank.
    Suppose that Algorithm~\ref{alg:proposed} is applied to $A$ and $\widehat X$ in exact arithmetic.
    If 
    \begin{align}
        \epsilon<\sqrt{\frac{\min_{1\leq j\leq k}|\lambda_{p_j}|}{\eta(\epsilon)\|A\|}}
    \end{align}
    then it holds
    \begin{align}
        \|E(k+1:n,j)-\widetilde E(k+1:n,j)\|&<\frac{\max_{k+1\leq i\leq n} |\lambda_{p_i}|\epsilon+\omega(\epsilon)\|A\|\epsilon^2}{\min_{1\leq j\leq k}|\lambda_{p_j}|-\eta(\epsilon)\|A\|\epsilon^2}.
    \end{align}
\end{corollary}

Combining Lemmas~\ref{lem:ogita2} and~\ref{lem:proposed2}, if $\min_{1\leq j\leq k}|\lambda_{p_j}|>0$, we have
\begin{align}
    \|E_L-\widetilde E_L\|&\leq \|E_{11}-\widetilde E_{11}\|+\|E_{21}-\widetilde E_{21}\|\\
    &=\frac{\max_{k+1\leq i\leq n} |\lambda_{p_i}|}{\min_{1\leq j\leq k}| \lambda_{p_j}|}\cdot\epsilon+\mathcal{O}(\epsilon^2).\label{eq:EL-tildeEL}
\end{align}
This bound yields the following theorem.

\begin{theorem}\label{thm:proposed}
    Let $A$ be a real symmetric $n\times n$ matrix with simple eigenvalues $\lambda_i$ ($i=1,\dots,n$).
    Let $p$ be a permutation of $\{1,\dots,n\}$.
    Define $X=(x^{(p_j)})_{j=1}^k\in\mathbb{R}^{n\times k}$ and assume $k<n$.
    Let $\widehat X\in\mathbb{R}^{n\times k}$ have full column rank.
    Suppose that Algorithm~\ref{alg:proposed} is applied to $A$ and $\widehat X$ in exact arithmetic.
    If 
    \begin{align}
        \epsilon:=\|E\|\leq \min
        \left(
        \sqrt{\frac{\min_{\substack{1\le i\le k\\ i\ne j}}|\lambda_{p_j}-\lambda_{p_i}|}{2\eta(\epsilon)k\|A\|}}, 
        \sqrt{\frac{|\lambda_{p_j}|}{\eta(\epsilon)\|A\|}}
        \right),\quad 1\leq j\leq k,
    \end{align}
    then it holds
    \begin{align}
        \frac{\|x^{(p_j)}-\widetilde x^{(p_j)}\|}{\|X-\widehat X\|}
        &\leq \frac{1}{1-2\epsilon}\Biggl(
        \frac{\omega(\epsilon)\sqrt{k}\|A\|\epsilon}{\min_{\substack{1\le i\le k\\ i\ne j}}|\lambda_{p_j}-\lambda_{p_i}|-2\eta(\epsilon)\|A\|\epsilon^2}\\
        &\qquad\qquad +
        \frac{\max_{k+1\leq i\leq n} |\lambda_{p_i}|+\omega(\epsilon)\|A\|\epsilon}{|\lambda_{p_j}|-\eta(\epsilon)\|A\|\epsilon^2}
        \Biggr).
    \end{align}
\end{theorem}

\begin{proof}
    From the definition of $E_L$ in $X=\widehat X+\widehat HE_L$, we have
	    \begin{align}
	        \|X-\widehat X\|=\|X-\widehat X-\widehat HE_L+\widehat HE_L\|\geq\sigma_{\min}(\widehat H)\cdot\|E_L\|
	    \end{align}
	    Also, from the definition of $\widetilde E_L$ in $\widetilde X=\widehat X+\widehat H\widetilde E_L$,
    \begin{align}
        \|x^{(p_j)}-\widetilde x^{(p_j)}\|&\leq \|x^{(p_j)}-\widehat x^{(p_j)}-\widehat H\widetilde e_j\|\leq \|\widehat H\|\cdot \|e_j-\widetilde e_j\|
    \end{align}
    holds. Therefore, by Lemma~\ref{lem:proposed1} and \eqref{eq:EL-tildeEL},
    \begin{align}
        \frac{\|x^{(p_j)}-\widetilde x^{(p_j)}\|}{\|X-\widehat X\|}&\leq \kappa_2(\widehat H)\frac{\|e_j-\widetilde e_j\|}{\|E_L\|}\\
        &\leq \frac{1}{1-2\epsilon}\Biggl(
        \frac{\omega(\epsilon)\sqrt{k}\|A\|\epsilon}{\min_{\substack{1\le i\le k\\ i\ne j}}|\lambda_{p_j}-\lambda_{p_i}|-2\eta(\epsilon)\|A\|\epsilon^2}\\
        &\qquad\qquad +
        \frac{\max_{k+1\leq i\leq n} |\lambda_{p_i}|+\omega(\epsilon)\|A\|\epsilon}{|\lambda_{p_j}|-\eta(\epsilon)\|A\|\epsilon^2}
        \Biggr)
    \end{align}
    follows.
\end{proof}

We now give a remark on linear convergence. When $\epsilon$ is sufficiently small,
\begin{align}
    \limsup_{\epsilon\to 0}\frac{\|x^{(p_j)}-\widetilde x^{(p_j)}\|}{\|X-\widehat  X\|}&\leq \frac{\max_{k+1\leq i\leq n} |\lambda_{p_i}|}{|\lambda_{p_j}|}
\end{align}
holds.
Moreover, using Corollary~\ref{cor:proposed2},
\begin{align}
    \epsilon\leq \min
    \left(
    \sqrt{\frac{\min_{\substack{1\le i\le k\\ i\ne j}}|\lambda_{p_j}-\lambda_{p_i}|}{2\eta(\epsilon)k\|A\|}}, 
    \sqrt{\frac{\min_{1\leq j\leq k}|\lambda_{p_j}|}{\eta(\epsilon)\|A\|}}
    \right),\quad 1\leq j\leq k,
\end{align}
we immediately obtain
\begin{align}
    \limsup_{\epsilon\to 0}\frac{\|X-\widetilde X\|}{\|X-\widehat  X\|}&\leq \frac{\max_{i>k} |\lambda_{p_i}|}{\min_{j\leq k}|\lambda_{p_j}|}.
\end{align}
Define
\[
\gamma := \frac{\max_{k+1\leq i\leq n} |\lambda_{p_i}|}{\min_{1\leq j\leq k}|\lambda_{p_j}|},\quad \min_{1\leq j\leq k}|\lambda_{p_j}|>0.
\]
Then a necessary condition for the convergence of Algorithm~\ref{alg:proposed} is $\gamma<1$.
When the method converges, the convergence rate is linear and depends on $\gamma$.
Section~\ref{sec:sufficient-conditions} records a conservative sufficient condition (allowing overestimation) and discusses how it motivates a notion of ``difficult'' instances, including cases where eigenvalues are not clustered but the initial approximation is too inaccurate for the refinement to behave as a local contraction.

\subsection{Sufficient Conditions}\label{sec:sufficient-conditions}

We next discuss a conservative sufficient condition for Algorithm~\ref{alg:proposed} to converge (allowing overestimation).

\begin{theorem}\label{thm:suffcond}
    Let $A$ be a real symmetric $n\times n$ matrix with simple eigenvalues $\lambda_i, i=1,2,\dots,n$, and let $p$ be a permutation of $\{1,2,\dots,n\}$.
    Define $X=(x^{(p_1)},x^{(p_2)},\dots,x^{(p_k)})\in\mathbb{R}^{n\times k}$ with $k<n$.
    For a given full column rank $\widehat X\in\mathbb{R}^{n\times k}$, suppose that Algorithm~\ref{alg:proposed} is applied to $A$ and $\widehat X$ in exact arithmetic.
    Suppose $\min_{1\leq j\leq k}|\lambda_{p_j}|>0$.
    Let $\rho_{1}$ and $\rho_{2}$ be positive scalars satisfying
    \begin{align}
        \epsilon=\frac{1}{\rho_1}\cdot\sqrt{\frac{\min_{i\not=j}|\lambda_{p_j}-\lambda_{p_i}|}{2\eta(\epsilon)\sqrt{k}\|A\|}}=
        \frac{1}{\rho_2}\cdot\sqrt{\frac{\min_{1\leq j\leq k}|\lambda_{p_j}|}{\eta(\epsilon)\|A\|}}.
    \end{align}
    If
    \begin{align}
        \epsilon<\frac{1}{100}\text{\quad and\quad }
        \frac{\max_{k+1\leq j\leq n} |\lambda_{p_j}|}{\min_{1\leq i\leq k}| \lambda_{p_i}|}
        &\leq\frac{\rho_2^2-1}{\rho_2^2}\left(\frac{98}{100}-\frac{46}{\rho_1^2-1}\right)-\frac{92}{\rho_2^2},
    \end{align}
    then it holds
    \begin{align}
        \frac{\|X-\widetilde X\|}{\|X-\widehat X\|}
        &<1.
    \end{align}
\end{theorem}
\begin{proof}

First, substituting
\begin{align}
    \epsilon=\frac{1}{\rho_1}\cdot\sqrt{\frac{\min_{i\not=j}|\lambda_{p_j}-\lambda_{p_i}|}{2\eta(\epsilon)\sqrt{k}\|A\|}}
\end{align}
into \eqref{eq:E11} gives
\begin{align}
    \frac{\omega(\epsilon)\sqrt{k}\|A\|\epsilon}{\min_{i\not=j}|\lambda_{p_j}-\lambda_{p_i}|-2\eta(\epsilon)\|A\|\epsilon^2}&=
    \frac{\frac{\omega(\epsilon)}{2\eta(\epsilon)}\epsilon}{\frac{\min_{i\not=j}|\lambda_{p_j}-\lambda_{p_i}|}{2\eta(\epsilon)\sqrt{k}\|A\|}-\frac{\epsilon^2}{\sqrt{k}}}
    \leq
    \frac{0.46\epsilon}{\epsilon^2(\rho_1^2-1)}\\
    &\leq \frac{46}{\rho_1^2-1}
	\end{align}

	Next, substituting
\begin{align}
    \epsilon=\frac{1}{\rho_2} \sqrt{\frac{\min_{1\leq i\leq k}|\lambda_{p_i}|}{\eta(\epsilon)\|A\|}}
\end{align}
into \eqref{eq:E21} yields
\begin{align}
    \frac{\frac{\max_{k+1\leq j\leq n} |\lambda_{p_j}|+\omega(\epsilon)\|A\|\epsilon}{\min_{1\leq i\leq k}| \lambda_{p_i}|}}{1-\frac{\eta(\epsilon)\|A\|}{\min_{1\leq i\leq k}| \lambda_{p_i}|}\epsilon^2}
    &=
    \frac{\frac{\max_{k+1\leq j\leq n} |\lambda_{p_j}|}{\min_{1\leq i\leq k}| \lambda_{p_i}|}+100\frac{\omega(\epsilon)}{\eta(\epsilon)\rho_2^2}}{1-\frac{1}{\rho_2^2}}\\
    &<
    \frac{1}{\rho_2^2-1}
	    \left(\rho_2^2\cdot\frac{\max_{k+1\leq j\leq n} |\lambda_{p_j}|}{\min_{1\leq i\leq k}| \lambda_{p_i}|}+91.5\right)
	\end{align}

	Therefore,
\begin{align}
    \frac{\|X-\widetilde X\|}{\|X-\widehat X\|}&<
    \frac{100}{98}\left(\frac{46}{\rho_1^2-1}+\frac{1}{\rho_2^2-1}
    \left(\rho_2^2\cdot\frac{\max_{k+1\leq j\leq n} |\lambda_{p_j}|}{\min_{1\leq i\leq k}| \lambda_{p_i}|}+91.5\right)\right)=:\beta
\end{align}
holds. 
Solving $\beta\leq1$ for ${\max_{k+1\leq j\leq n} |\lambda_{p_j}|}/{\min_{1\leq i\leq k}| \lambda_{p_i}|}$ gives
\begin{align}
	\frac{\max_{k+1\leq j\leq n} |\lambda_{p_j}|}{\min_{1\leq i\leq k}| \lambda_{p_i}|}
    \leq
    \frac{\rho_2^2-1}{\rho_2^2}\left(\frac{98}{100}-\frac{46}{\rho_1^2-1}\right)-\frac{92}{\rho_2^2}=:\alpha.\label{eq:alpha}
\end{align}

\end{proof}

Theorem~\ref{thm:suffcond} indicates that the proposed method (Algorithm~\ref{alg:proposed}) is applicable if at least $\alpha<1$, and larger $\alpha$ is better for the proposed method.
Table~\ref{tab:a} lists the resulting values of $\alpha$ for representative choices of $\rho_1$ and $\rho_2$.
\begin{table}[htbp]
\centering
\caption{The values of $\alpha$; sufficient condition for the convergence of the proposed method}
\label{tab:a}
\begin{tabular}{rcccc}
\toprule
$\rho_1\backslash \rho_2$ & 10 & 100 & 1,000 & 10,000\\
\midrule
10     & (-0.4144) & 0.5061 & 0.5153 & 0.5154\\
100    & 0.0456 & 0.9661 & 0.9753 & 0.9754\\
1,000  & 0.0502 & 0.9707 & 0.9799 & 0.9800\\
10,000 & 0.0502 & 0.9707 & 0.9799 & 0.9800\\
\bottomrule
\end{tabular}
\end{table}

When $\rho_1=\rho_2=10$, $\alpha$ becomes negative, so the sufficient condition cannot be satisfied.
If either $\rho_1=10$ or $\rho_2=10$, the sufficient condition suggests that only a very limited class of problems can be handled.
In contrast, when $\rho_1,\rho_2\geq 100$, the bound indicates that the proposed method can still be applicable even when ${\max_{k+1\leq j\leq n} |\lambda_{p_j}|}/{\min_{1\leq i\leq k}| \lambda_{p_i}|}$ is close to $1$.

It suggests that refinement can fail not only for tightly clustered targeted eigenvalues but also when the initial approximation is too inaccurate (large $\epsilon$), even if the targeted eigenvalues are well separated.
\subsection{Handling Clustered Eigenvalues}

We consider the case where $A$ has clustered eigenvalues, i.e., $\min_{1\leq i\ne j\leq k}|\lambda_{p_j}-\lambda_{p_i}|$ is extremely small.
The proposed framework can incorporate the strategies in~\cite{ogita2019iterative,shiroma2019tracking}.
To avoid unnecessary technicalities, we describe a simplified version of the approach in~\cite{shiroma2019tracking}.

Given an approximate eigenvector matrix $\widehat X\in\mathbb{R}^{n\times k}$, consider the generalized eigenvalue problem
\begin{align}
    (\widehat X^\mathsf{T}A\widehat X)S=(\widehat X^\mathsf{T}\widehat X)S\widetilde D
\end{align}
and set $Z\leftarrow \widehat XS$. Then
\begin{align}
    Z^\mathsf{T}Z=I_k,\quad Z^\mathsf{T}AZ=\widetilde D_1.
\end{align}

Using this property, we update $\widehat X\leftarrow Z$ and apply Algorithm~\ref{alg:proposed}, which yields
\begin{align}
    \widetilde e_{ij}=
    \begin{cases}
        0, & 1\leq i,j\leq k\\
        v_{ij}/\widetilde d_j, & k+1\leq i\leq n, 1\leq j\leq k
    \end{cases}
    \end{align}
That is, the refinement step avoids denominators of the form $\widetilde d_j-\widetilde d_i$.

In practice, we compute an approximation $\widehat S$ of $S$, so a separate analysis is required to quantify how tightly clustered eigenvalues can be handled.
Moreover, when $k$ is large, it is necessary---as in~\cite{shiroma2019tracking}---to extract only the eigenvectors associated with the clustered eigenvalues and solve the Rayleigh--Ritz equations on that subspace.

\section{Numerical Experiments}\label{sec:numerical-experiments}

\subsection{Experimental Setup}

All numerical computations were executed using MATLAB~R2025b.

Let $K$ denote the number of eigenvectors used in both the initial approximation and the iterative refinement.
In contrast, convergence is judged using the size of the correction matrix for the leading $k$ eigenvectors.
The approximate eigenvectors are computed by
\begin{verbatim}
    [X,D] = eigs(single(A),K); X = double(X);
\end{verbatim}
which computes the $K$ eigenpairs with the largest absolute eigenvalues in single precision and then converts the eigenvectors to double precision.
All computations in the iterative refinement are carried out in double precision to improve the accuracy of these $K$ eigenvectors.

This MATLAB implementation is intended primarily to validate convergence behavior and attainable accuracy.
Since it is not tuned for performance (e.g., optimized sparse matrix--matrix kernels and accelerator offloading are not exploited), we do not attempt to draw hardware-independent performance conclusions.
The main motivation is scenarios where a low-precision or low-cost initial solve is available and only a subset of eigenvectors requires additional accuracy.

For large-scale problems, the exact eigenvectors are not available, so we evaluate accuracy primarily through residual-based quantities.
In particular, we report the relative residual norm $\|A\widehat X-\widehat X\widetilde D_1\|_F/\|A\|$ and the correction size $\|\widetilde E_L\|_F$.
For symmetric problems, Davis--Kahan type bounds relate residual norms to forward error through eigenvalue gaps; when eigenvalues are tightly clustered, it is more meaningful to compare invariant subspaces (e.g., via principal angles) than individual eigenvector bases~\cite{davis1970rotation}.

We introduce $K$ because, when $\epsilon$ is sufficiently small (see Section~\ref{sec:convergence-analysis}),
\begin{align}
    \frac{\|x^{(p_j)}-\widetilde x^{(p_j)}\|}{\|X-\widehat X\|}
    \approx
    \frac{\max_{K+1\leq i\leq n}|\lambda_{p_i}|}{|\lambda_{p_j}|}
    \le
    \frac{\max_{k+1\leq i\leq n}|\lambda_{p_i}|}{|\lambda_{p_j}|},
\end{align}
so taking a larger $K$ can potentially accelerate convergence.
On the other hand, a larger $K$ increases the per-iteration cost; therefore, there is a trade-off between iteration count and total computation time.

To complement the residual-based evaluation on large-scale sparse problems, we report convergence histories for two representative matrices (Figure~\ref{fig:conv:sparse}) and summarize iteration counts and correction norms for several SuiteSparse matrices (Table~\ref{tab:sparse1} and Table~\ref{tab:sparse1_precond}).
For a smaller dense test problem where a reference invariant subspace can be computed by \texttt{eig}, we assess forward error via the sine of the largest principal angle between the computed and reference subspaces (Table~\ref{tab:dense_forward_error}).

\subsection{Dense matrices}

A real symmetric matrix $A$ is generated using MATLAB's built-in \texttt{gallery} routine (option \texttt{'randsvd'}).
The resulting matrix has a condition number of approximately $10^5$ and a spectral norm of about $1$.

\begin{figure}[htbp]
    \centering
    \subfloat[\texttt{mode=3}]{
        \begin{minipage}[b]{0.45\textwidth}
            \centering
            \includegraphics[width=\linewidth]{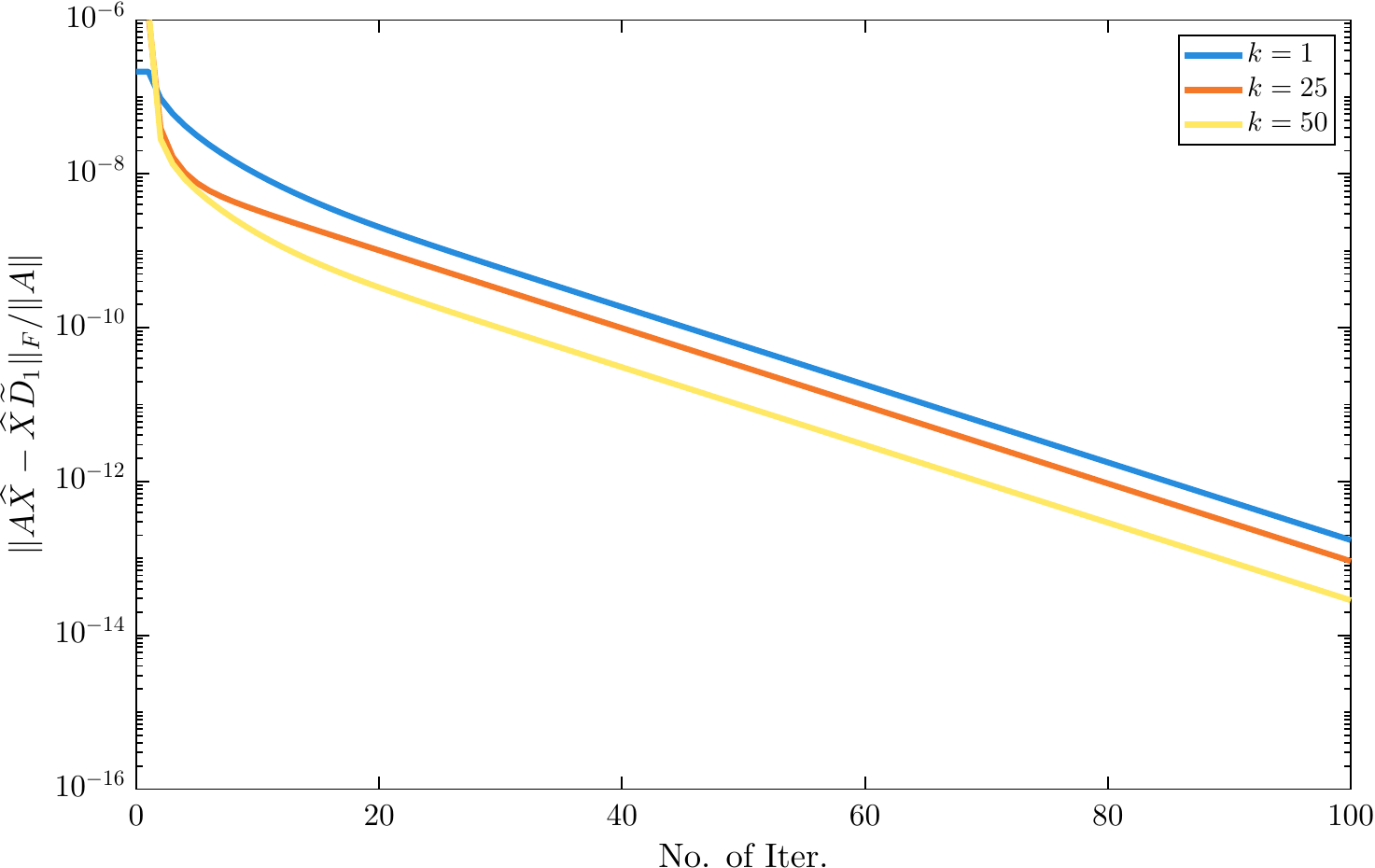}
        \end{minipage}
    }
    \hfill
    \subfloat[\texttt{mode=4}]{
        \begin{minipage}[b]{0.45\textwidth}
            \centering
            \includegraphics[width=\linewidth]{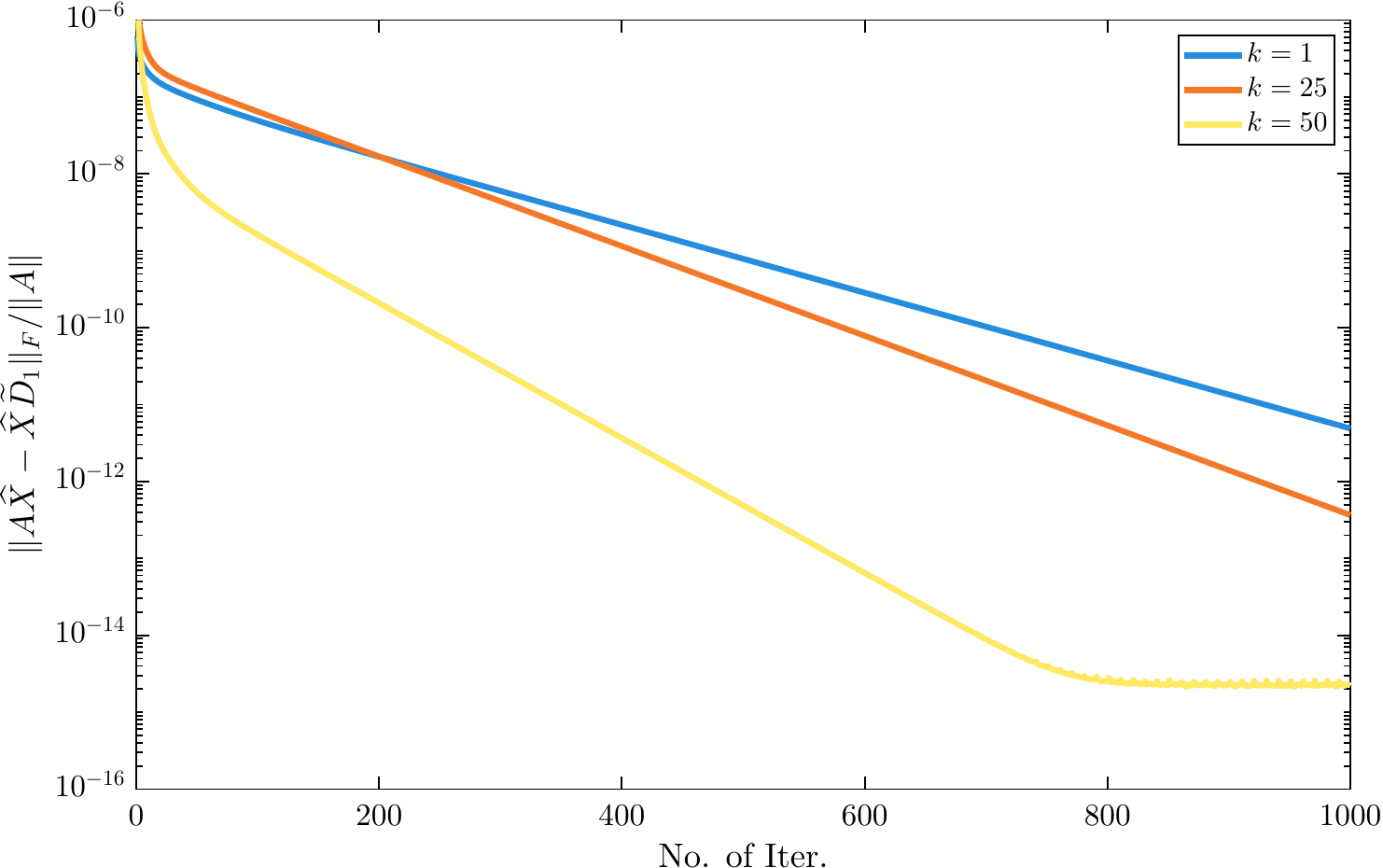}
        \end{minipage}
    }
    \caption{Convergence history of the relative residual norm $\|A\widehat X-\widehat X\widetilde D_1\|_F/\|A\|$ with $k=1, 25$, and $50$ for \texttt{n=100} (with $K=k$). Subfigures (a) and (b) correspond to \texttt{mode=3} and \texttt{mode=4}, respectively.}
    \label{fig:sym:conv:K}
\end{figure}

Figure~\ref{fig:sym:conv:K} shows the convergence history of the proposed method when the number of refined eigenvectors $k$ is varied (with $K=k$).
For \texttt{mode=3}, larger $k$ reduces the iteration count, while the asymptotic slope is essentially unchanged.
For \texttt{mode=4}, larger $k$ again reduces the iteration count and also leads to a modest improvement in the convergence rate.

Next, we examine how the choice of $K$ affects convergence for the leading $k=5$ eigenpairs $(\widehat{\lambda}_{p_j}, \widehat{x}^{(p_j)})$ ($j=1,\dots,5$).
Figure~\ref{fig:sym:conv:mode4} shows the per-eigenpair residual histories for $K\in\{5,10,20\}$.
For both modes, increasing $K$ tends to improve convergence for those eigenpairs whose eigenvalue gaps are smaller, consistent with the role of the $\max_{i\ge K+1}|\lambda_{p_i}|/|\lambda_{p_j}|$ factor in the analysis.

\begin{figure}[htbp]
    \centering
    % Use a tabular layout to ensure a 3-by-2 grid (avoid line wrapping).
    \setlength{\tabcolsep}{2pt}
    \begin{tabular}{@{}ccc@{}}
        \subfloat[$K=5$, \texttt{mode=3}]{\includegraphics[width=0.28\textwidth]{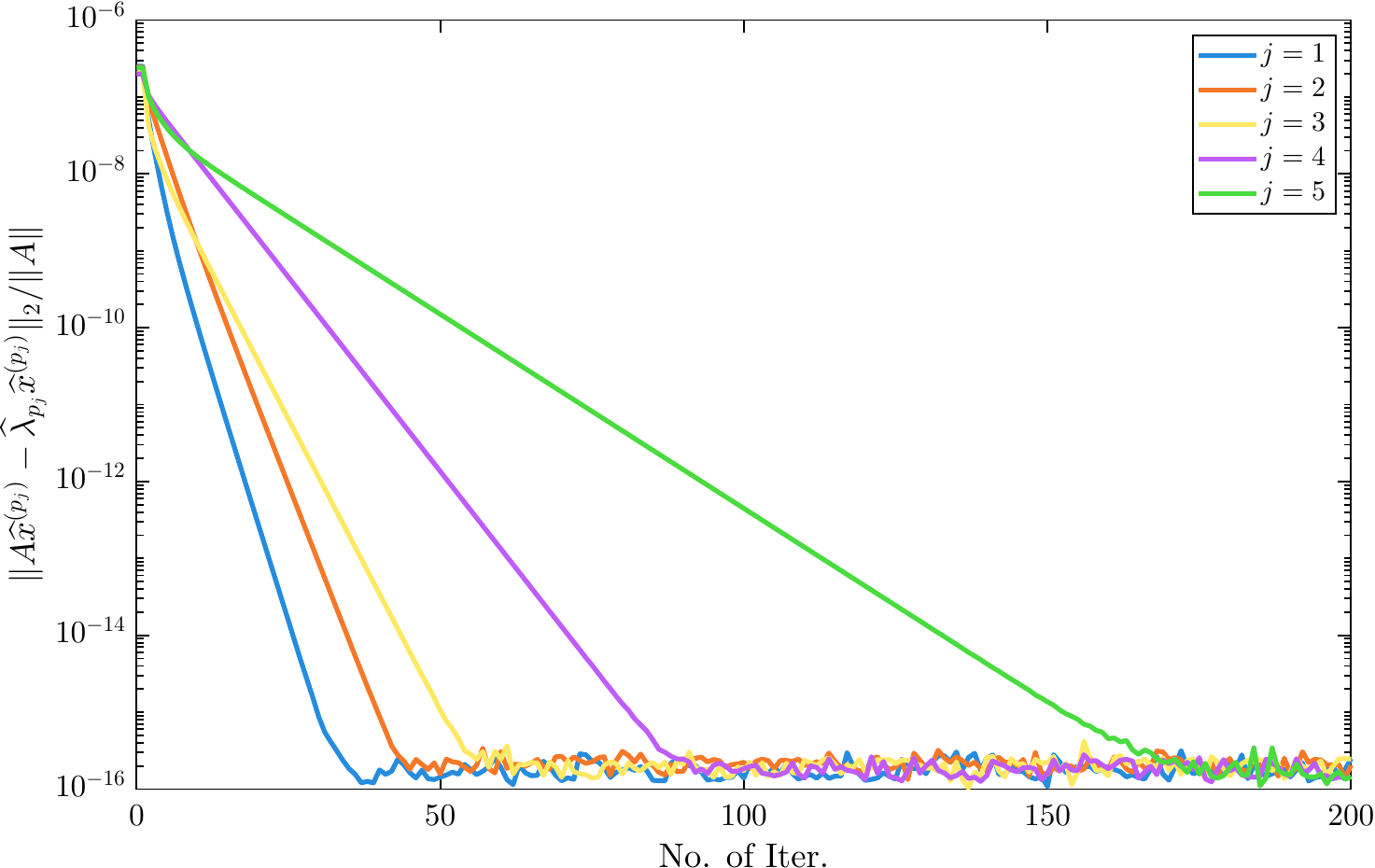}} &
        \subfloat[$K=10$, \texttt{mode=3}]{\includegraphics[width=0.28\textwidth]{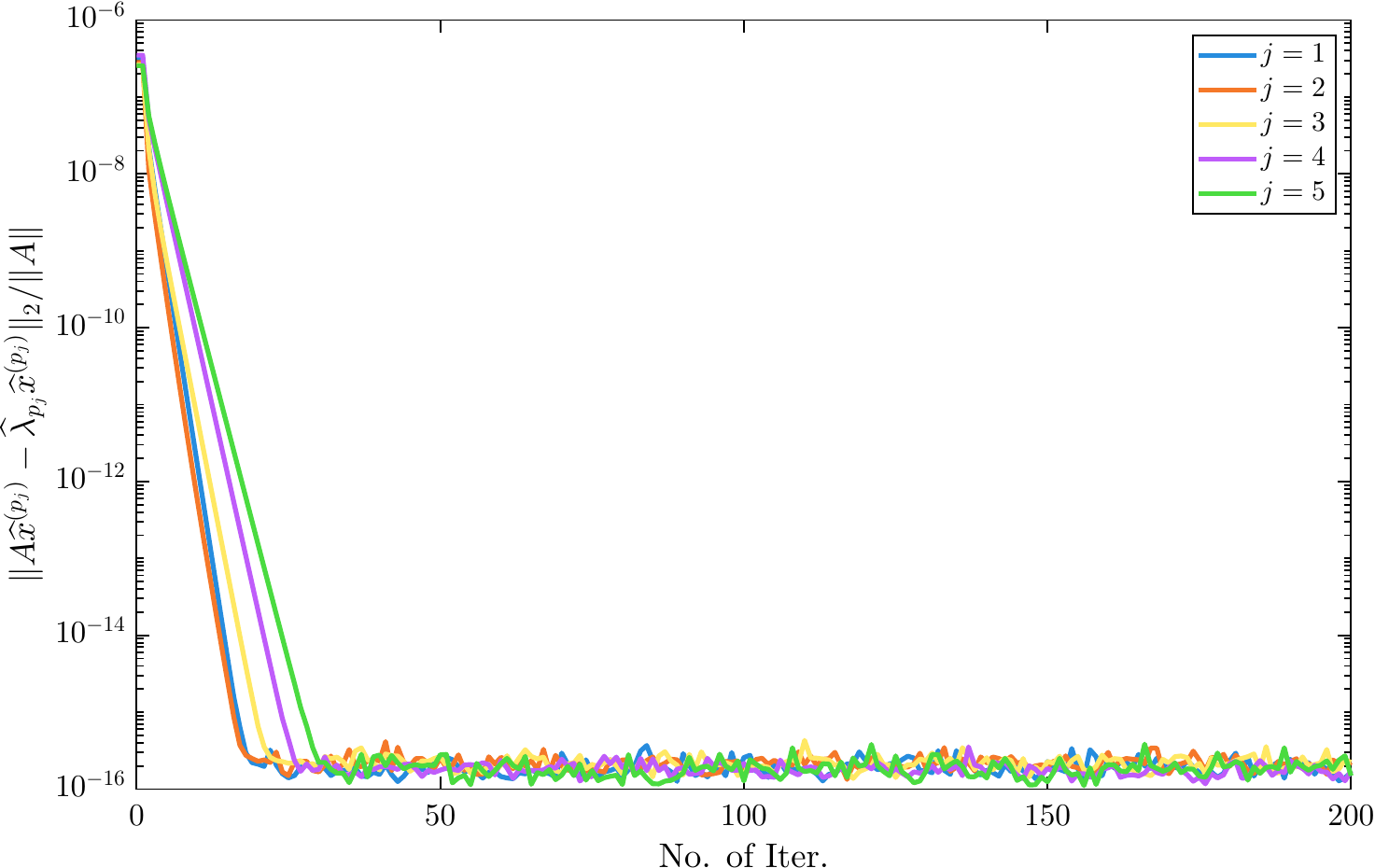}} &
        \subfloat[$K=20$, \texttt{mode=3}]{\includegraphics[width=0.28\textwidth]{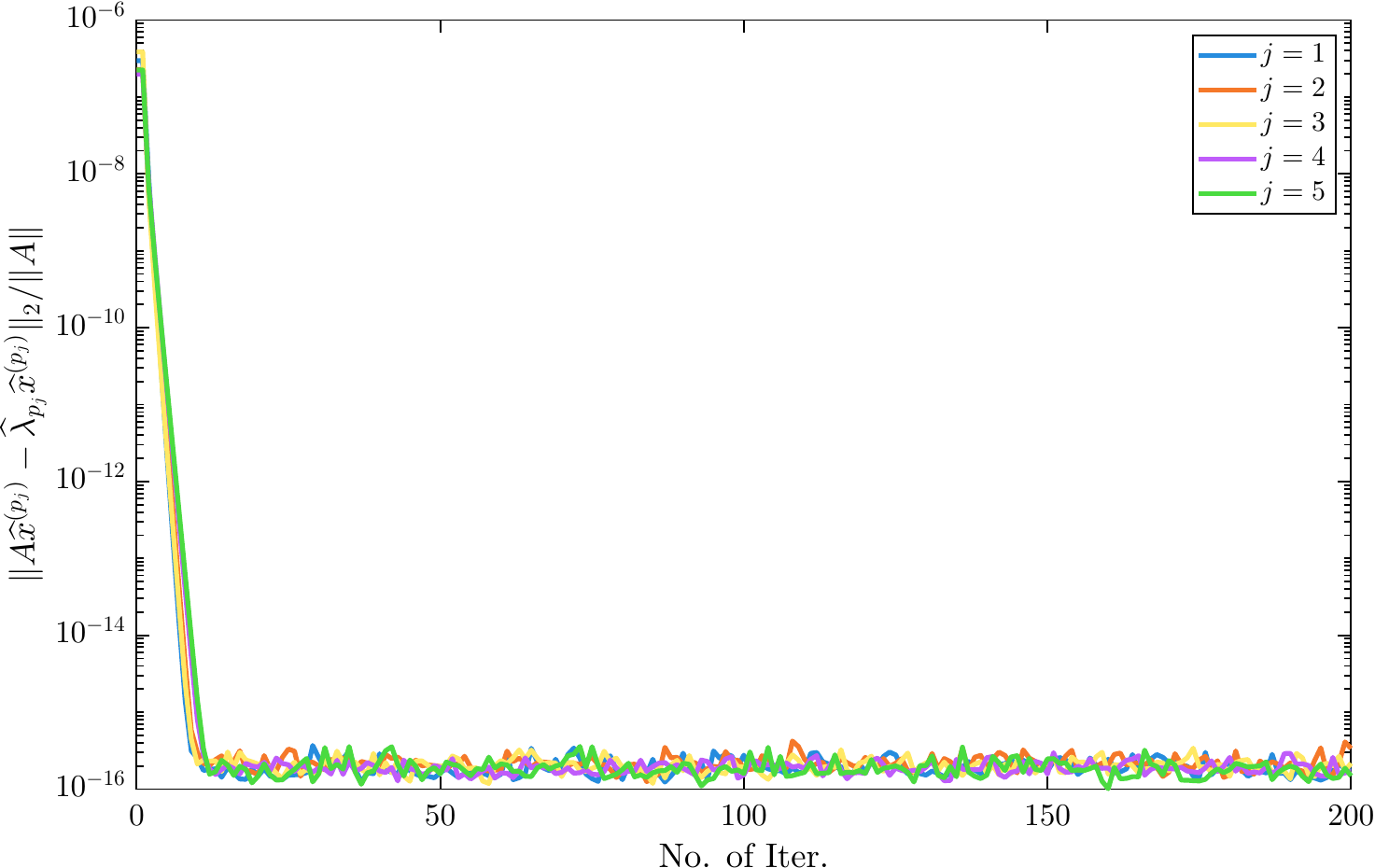}} \\
        \subfloat[$K=5$, \texttt{mode=4}]{\includegraphics[width=0.28\textwidth]{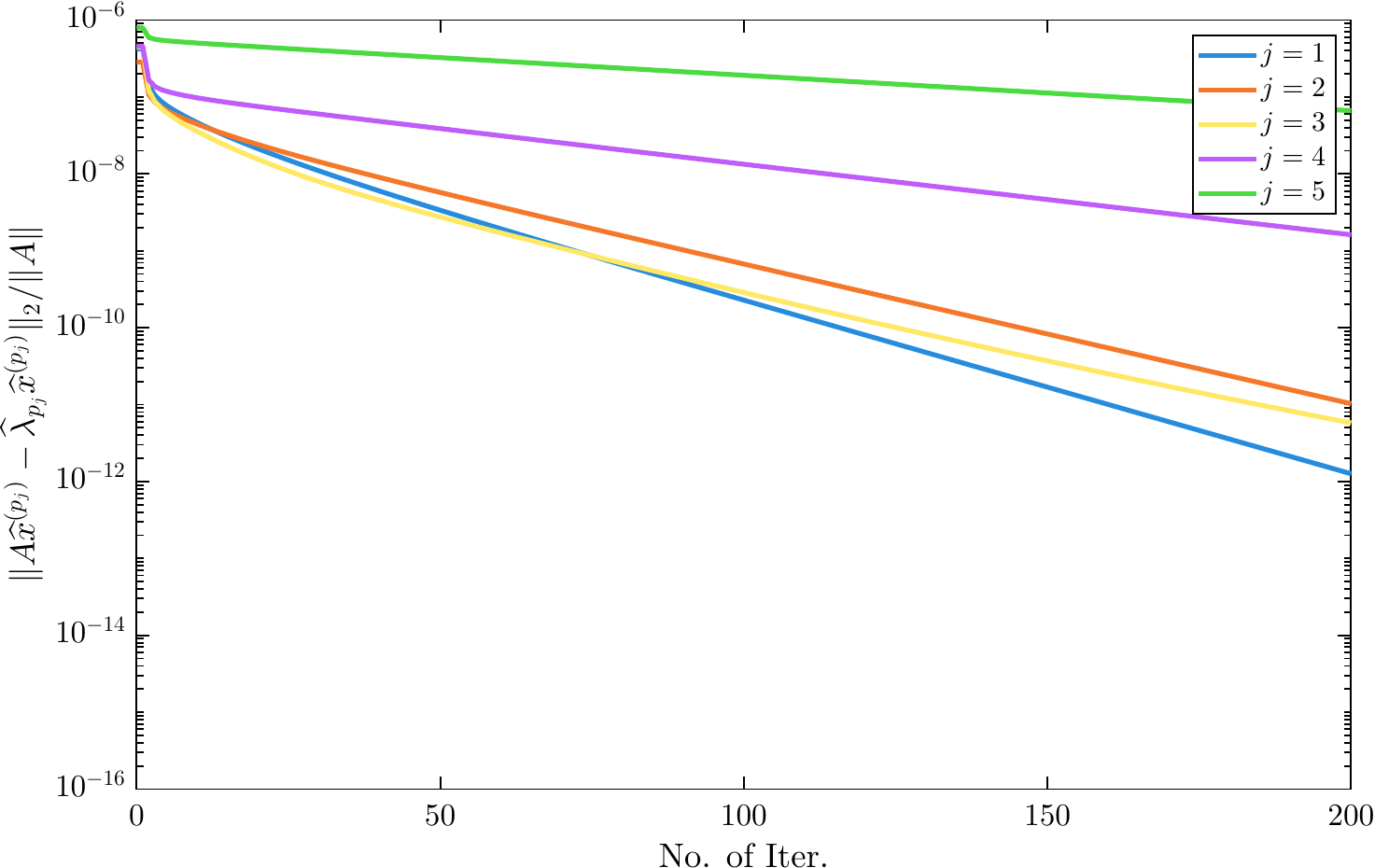}} &
        \subfloat[$K=10$, \texttt{mode=4}]{\includegraphics[width=0.28\textwidth]{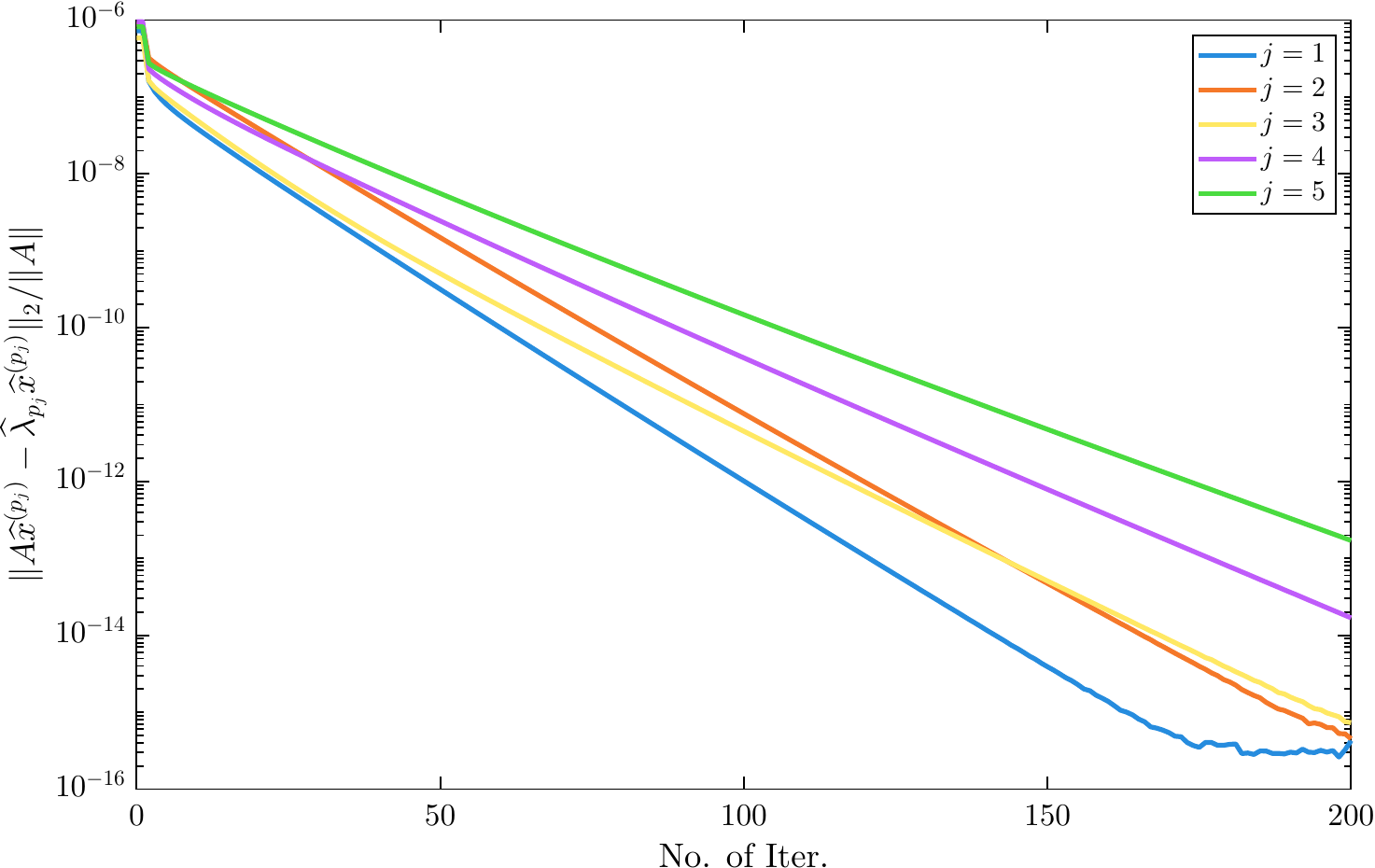}} &
        \subfloat[$K=20$, \texttt{mode=4}]{\includegraphics[width=0.28\textwidth]{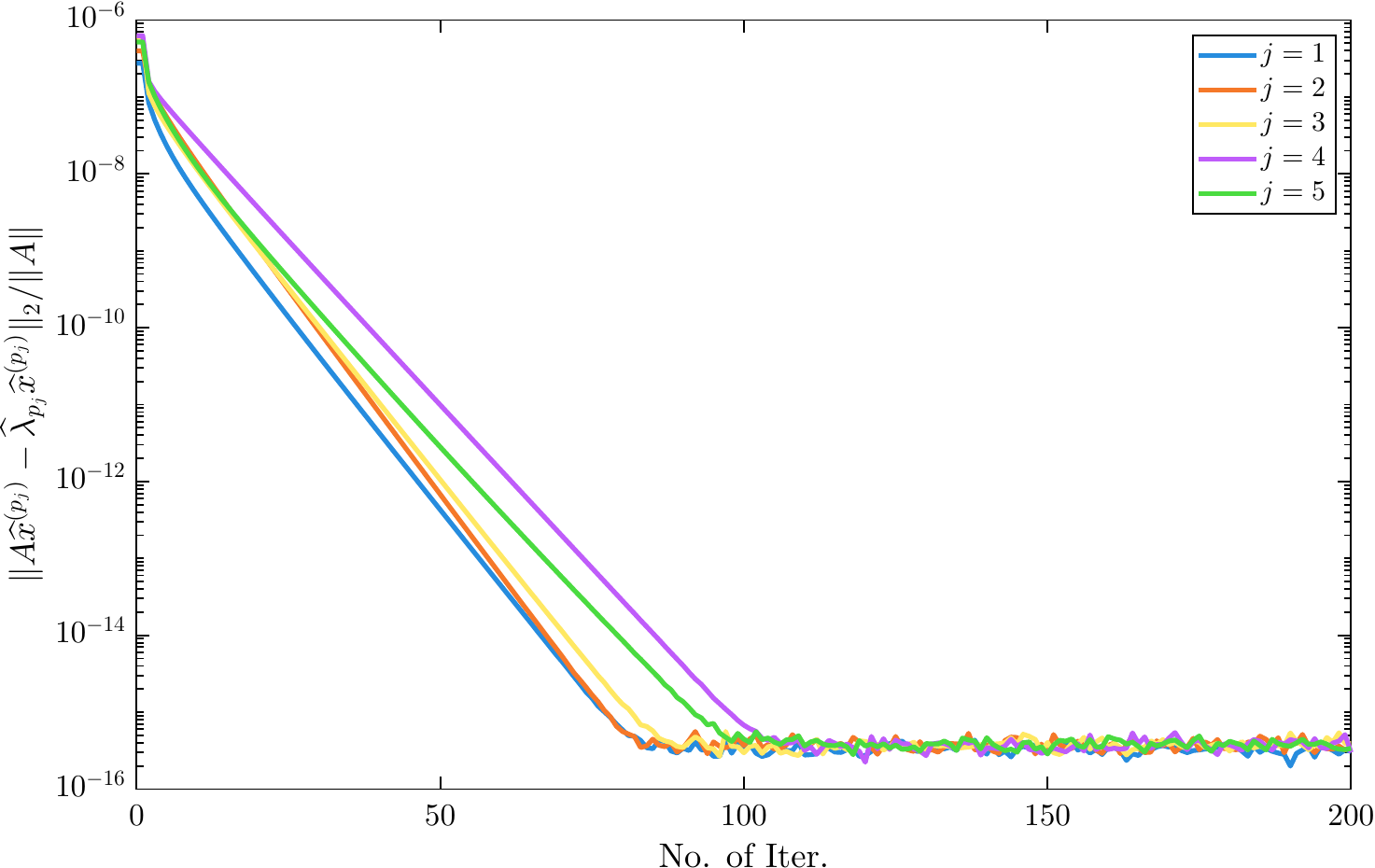}} \\
    \end{tabular}
    \caption{Convergence history of the relative eigenpair residual $\|A\widehat x^{(p_j)}-\widehat \lambda_{p_j}\widehat x^{(p_j)}\|_2/\|A\|$ for $j=1,\dots,5$ (\texttt{n=100}). Each panel is labeled with $(K,\texttt{mode})$.}
    \label{fig:sym:conv:mode4}
\end{figure}

In these dense-matrix experiments, when the refinement converges we observe that the eigenpair residuals are reduced to the level of double precision:
\begin{align}
    \|A\widehat x^{(p_j)}-\widehat \lambda_{p_j}\widehat x^{(p_j)}\|_2 = \mathcal{O}(u)\|A\|,
    \qquad
    u=2^{-53}.
\end{align}
This suggests that, for sufficiently well-separated eigenvalues, the attainable accuracy is ultimately limited by double-precision rounding.

To assess eigenvector forward error on a problem where a reference is available, we consider a smaller dense problem with \texttt{n=500} and compute a reference invariant subspace using \texttt{eig} in double precision.
Table~\ref{tab:dense_forward_error} reports the relative residual norm and $\sin\Theta$, where $\Theta$ is the largest principal angle between the computed and reference $k$-dimensional invariant subspaces.
We define $N_{\mathrm{match}}$ as the first iteration index at which the refined relative residual falls below the double-precision \texttt{eigs} residual level.
For the well-separated case (\texttt{mode=3}), refinement reaches the double-precision residual level in a moderate number of iterations and substantially decreases $\sin\Theta$.
For the more clustered case (\texttt{mode=4}), refinement can still attain the double-precision residual level, but it requires significantly more iterations, illustrating the slower convergence expected when the target eigenvalues are tightly clustered.

\begin{table}[htbp]
\centering
\small
\caption{Forward-error assessment on a small dense test problem (\texttt{n=500}, $k=5$). ``Single'' and ``double'' refer to the \texttt{eigs} outputs in single and double precision; refinement is run until the refined relative residual reaches the double-precision \texttt{eigs} residual level (after $N_{\mathrm{match}}$ iterations). The reference invariant subspace is computed by \texttt{eig} in double precision.}
\label{tab:dense_forward_error}
\begin{tabular}{lrrrrrrr}
\toprule
 & & \multicolumn{3}{c}{Relative residual} & \multicolumn{3}{c}{$\sin\Theta$}\\\cmidrule(lr){3-5}\cmidrule(lr){6-8}
Mode & $N_{\mathrm{match}}$ & single & refined & double & single & refined & double\\
\midrule
\texttt{mode=3} & 760 & 6.85e-07 & 2.29e-15 & 2.36e-15 & 5.13e-06 & 9.54e-14 & 1.29e-14\\
\texttt{mode=4} & 7227 & 1.66e-06 & 4.41e-15 & 4.45e-15 & 4.89e-05 & 2.04e-12 & 1.29e-13\\
\bottomrule
\end{tabular}
\end{table}

\subsection{Sparse matrices}

We report results for several matrices from the SuiteSparse collection \cite{davis2011uf}.
Table~\ref{tab:sparse1} summarizes the results for five matrices.
We declare convergence when
\begin{align}
    \|\widetilde E_L\|_F\leq 10^{-8}
\end{align}
is satisfied.

All sparse experiments reported in this section were performed \emph{without} preconditioning unless otherwise noted.

\begin{table}[htbp]
\centering
\caption{Numerical results for matrices from the SuiteSparse Matrix Collection ($k=5$, $K=10$, no Rayleigh--Ritz preprocessing). The iteration count is the number of refinement steps; for \texttt{bmwcra\_1}, refinement was terminated early due to divergence.}
\label{tab:sparse1}
\begin{tabular}{lrrrrr}
\toprule
Matrix & Size & nnz & Iterations & $\|\widetilde E_L^{(0)}\|_F$ & $\|\widetilde E_L^{(\mathrm{end})}\|_F$ \\
\midrule
\texttt{bcsstk13} &
2,003 & 83,883 & 6 & 3.49e-05 & 5.37e-09\\
\texttt{bmwcra\_1} &
148,770 & 10,641,602 & 2 & 7.57e-01 & 1.03e+01\\
\texttt{gearbox} &
153,746 & 9,080,404 & 272 & 9.40e-04 & 9.90e-09\\
\texttt{nd3k} &
9,000 & 3,279,690 & 204 & 7.05e-04 & 9.96e-09\\
\texttt{shipsec5} &
179,860 & 4,598,604 & 120 & 9.45e-04 & 9.72e-09\\
\bottomrule
\end{tabular}
\end{table}

To provide a rough runtime comparison, Table~\ref{tab:sparse_timing5} reports wall-clock times of \texttt{eigs} in single and double precision and the additional time of refinement starting from the single-precision output until the refined relative residual matches the double-precision \texttt{eigs} residual level.
Unlike Table~\ref{tab:sparse1}, which reports iteration counts under the correction-size stopping criterion, Table~\ref{tab:sparse_timing5} uses this residual-matching criterion to define $N_{\mathrm{match}}$.
These timings are indicative only and depend on the implementation and hardware.
For \texttt{bmwcra\_1} without preprocessing, the refinement experienced numerical breakdown after only a few iterations (before reaching the double-precision residual level), so $N_{\mathrm{match}}$ is not defined and the reported refinement time corresponds to this early termination.

\begin{table}[htbp]
\centering
\small
\caption{Timing comparison for the five SuiteSparse matrices in Table~\ref{tab:sparse1} ($k=5$, $K=10$). $N_{\mathrm{match}}$ denotes the first iteration index at which the refined relative residual falls below the double-precision \texttt{eigs} residual level. Times are in seconds.}
\label{tab:sparse_timing5}
\begingroup
\begin{tabular}{lrrrrr}
\toprule
Matrix &
$t_{\mathrm{eigs}}^{\mathrm{single}}$ &
$t_{\mathrm{eigs}}^{\mathrm{double}}$ &
$N_{\mathrm{match}}$ &
$t_{\mathrm{ref}}$ &
$t_{\mathrm{eigs}}^{\mathrm{single}}+t_{\mathrm{ref}}$\\
\midrule
\texttt{bcsstk13}  & 0.020 & 0.204 & 26   & 0.030  & 0.049\\
\texttt{bmwcra\_1} & 0.746 & 1.918 & --   & 0.109  & 0.855\\
\texttt{gearbox}   & 0.925 & 2.590 & 1230 & 27.333 & 28.259\\
\texttt{nd3k}      & 0.197 & 0.626 & 1136 & 5.168  & 5.365\\
\texttt{shipsec5}  & 0.211 & 0.694 & 629  & 11.779 & 11.989\\
\bottomrule
\end{tabular}
\endgroup
\end{table}

From Table~\ref{tab:sparse1}, convergence was observed for all matrices except \texttt{bmwcra\_1}.
The iteration counts are problem-dependent and should be interpreted only as indicative reference values, since they depend strongly on the choices of $k$ and $K$, the stopping criterion, and implementation details.

The failure of \texttt{bmwcra\_1} is consistent with the theory: the target eigenvalues are tightly clustered and the initial approximation is not sufficiently accurate, so the refinement can stagnate or diverge.
In our run, the correction norm grew quickly and the refinement was terminated after only a few iterations.
Table~\ref{tab:sparse1_precond} shows that applying a preprocessing step can restore convergence for this case.
This suggests that the preprocessing step can be effective even when the initial approximation is poor or when clustered eigenvalues are present.

\begin{table}[htbp]
\centering
\caption{Numerical results for \texttt{bmwcra\_1} with preprocessing ($k=5$, $K=10$).}
\label{tab:sparse1_precond}
\begin{tabular}{lrrrrr}
\toprule
Matrix & Size & nnz & Iterations & $\|\widetilde E_L^{(0)}\|_F$ & $\|\widetilde E_L^{(\mathrm{end})}\|_F$ \\
\midrule
\texttt{bmwcra\_1} &
148,770 & 10,641,602 & 572 & 7.57e-01 & 9.99e-09\\
\bottomrule
\end{tabular}
\end{table}

\subsubsection{Limitations and practical guidance}

The refinement becomes challenging when the target eigenvalues are tightly clustered, i.e., when the minimal separation
\[
\min_{1\le i\ne j\le k}|\lambda_{p_j}-\lambda_{p_i}|
\]
decreases and when the ratio $\max_{k+1\le j\le n} |\lambda_{p_j}|/\min_{1\le i\le k} |\lambda_{p_i}|$ approaches one.
In this regime, the Rayleigh quotient differences can be comparable to rounding errors and refinement may stagnate or diverge.
In such cases, it is often more appropriate to assess accuracy in terms of invariant subspaces rather than individual eigenvector bases.
Empirically, for \texttt{bmwcra\_1} a preprocessing step based on Rayleigh--Ritz equations improved the diagonality and orthogonality of the approximation and restored convergence.
In practice, a simple and robust workflow is to (i) compute a low-precision initial approximation, (ii) apply Rayleigh--Ritz preprocessing on the clustered subset when needed, and (iii) run the refinement with a correction-size stopping criterion.
If $\|\widetilde E_L\|_F$ grows rapidly in the first few iterations, it is often preferable to terminate and switch to preprocessing rather than continuing the refinement.

\begin{figure}[htbp]
    \centering
    \subfloat[\texttt{gearbox}]{
        \begin{minipage}[b]{0.45\textwidth}
            \centering
            \includegraphics[width=\linewidth]{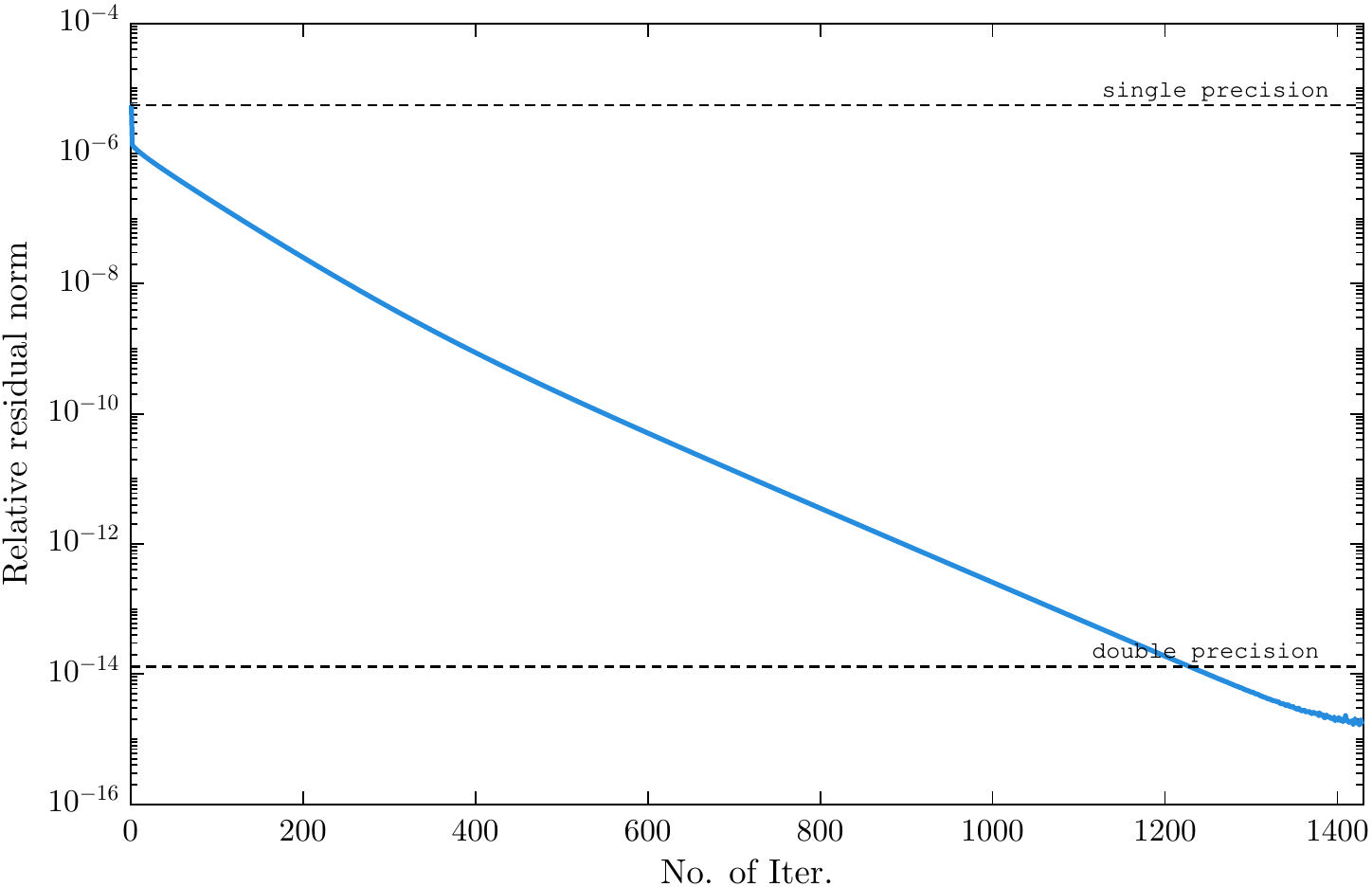}
        \end{minipage}
    }
    \hfill
    \subfloat[\texttt{shipsec5}]{
        \begin{minipage}[b]{0.45\textwidth}
            \centering
            \includegraphics[width=\linewidth]{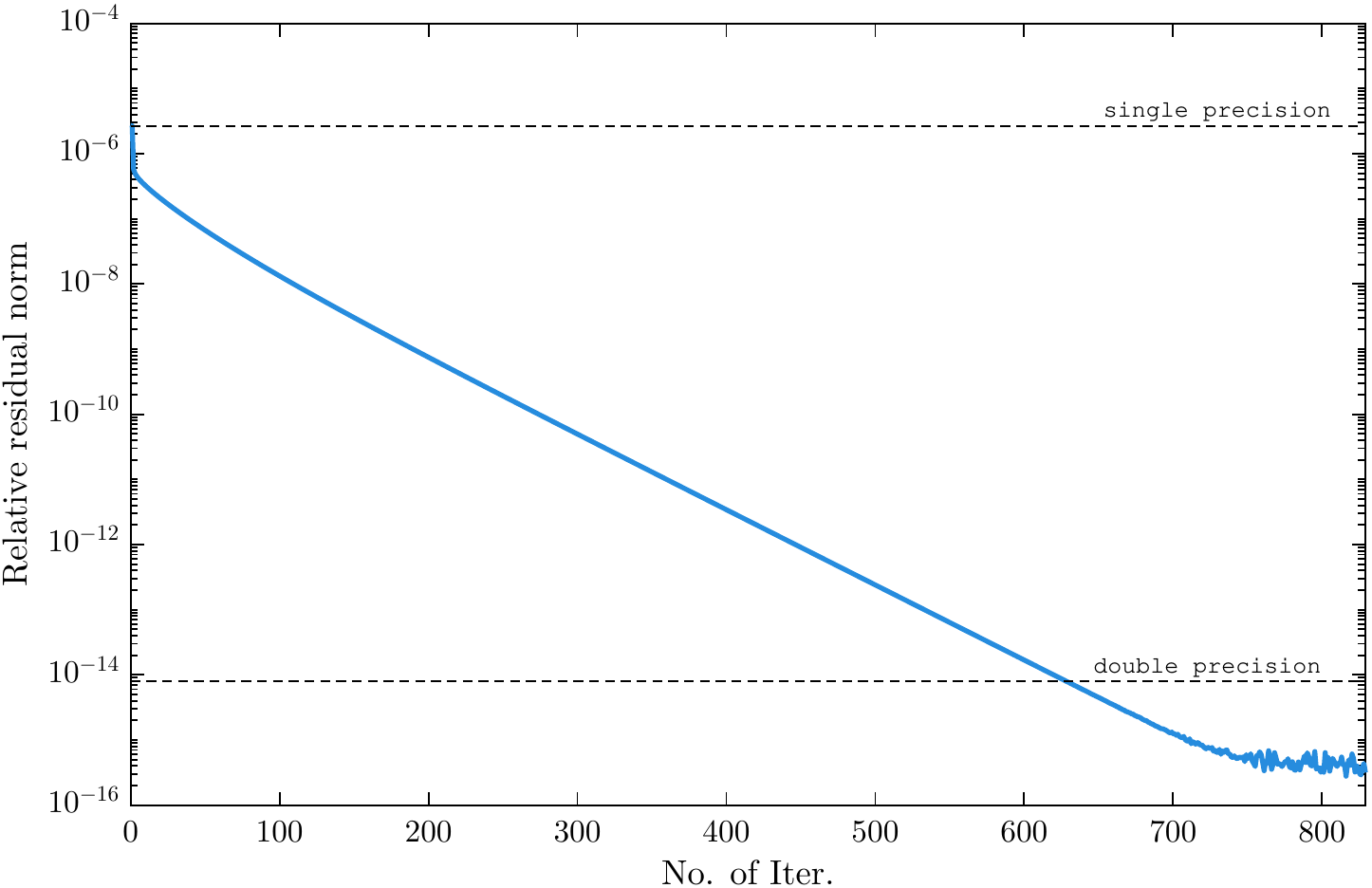}
        \end{minipage}
    }
    \caption{Convergence history of the relative residual norm $\|A\widehat X-\widehat X\widetilde D_1\|_F/\|A\|$ with $k=5$ and $K=10$. Dashed horizontal lines indicate the residual levels of the initial single-precision and double-precision \texttt{eigs} outputs.}
    \label{fig:conv:sparse}
\end{figure}

Figure~\ref{fig:conv:sparse} shows the convergence history of Algorithm~\ref{alg:proposed} for two representative matrices.
In both cases, the residual decreases steadily over the refinement iterations.
In practical runs, the refinement can be terminated once the correction-size criterion is satisfied (as in Table~\ref{tab:sparse1}); in Figure~\ref{fig:conv:sparse}, we plot the residual history over many refinement steps to visualize the long-term behavior.
Since the stopping criterion is user-controlled, the refinement can also be terminated early to trade cost for accuracy and obtain eigenvectors whose residuals lie between those returned by the initial single-precision solve and those obtained after full refinement.

\section{Variants of the Proposed Method}\label{sec:variants}

So far, we have considered a permutation $p$ such that, for the eigenvalues $\lambda_i$ of a real symmetric matrix $A$,
\begin{align}
    \max_{k+1\leq i\leq n}|\lambda_{p_i}|<\min_{1\leq j\leq k}|\lambda_{p_j}|
    %|\lambda_{p_1}|\geq |\lambda_{p_2}|\geq\dots\geq |\lambda_{p_n}|
\end{align}
holds.
This idea can be extended to other orderings and related problems.

\subsection{Smallest eigenvalues (and similarly largest eigenvalues)}

Let $\{p_i\}=\{1,\dots,n\}$ be defined by
\begin{align}
    \min_{k+1\leq i\leq n}\lambda_{p_i}>\max_{1\leq j\leq k}\lambda_{p_j}
    %\lambda_{p_1}\leq \lambda_{p_2}\leq \dots\leq \lambda_{p_n}
\end{align}
and consider iterative refinement for an approximation of $X=(x^{(p_1)},x^{(p_2)},\dots,x^{(p_k)})$.

We apply the proposed method to the shifted matrix $B=A-\alpha I$.
The eigenvalues of $B$ are $\lambda_i-\alpha$, and the corresponding eigenvectors are $x^{(i)}$.
If $\alpha$ is chosen such that
\begin{align}
    \max_{k+1\leq j\leq n}|\lambda_{p_j}-\alpha| < \min_{1\leq i\leq k}|\lambda_{p_i}-\alpha|
\end{align}
then the proposed method applied to $B$ satisfies the convergence condition.

A practical choice of $\alpha$ is to first compute an approximation $\widehat \lambda_{p_{k+1}}$ of $\lambda_{p_{k+1}}$ and set
\begin{align}
    (\|A\|+\widehat \lambda_{p_{k+1}})/2\leq (\|A\|_{\infty}+\widehat \lambda_{p_{k+1}})/2=\alpha.
\end{align}

Iterative refinement for eigenvectors associated with the $k$ largest eigenvalues can be handled similarly by applying an appropriate positive shift.

\subsection{Smallest eigenvalues with absolute values}

Let $\{p_i\}=\{1,\dots,n\}$ be defined by
\begin{align}
    \min_{k+1\leq i\leq n}|\lambda_{p_i}|>\max_{1\leq j\leq k}|\lambda_{p_j}|
    %|\lambda_{p_1}|\leq |\lambda_{p_2}|\leq\dots\leq |\lambda_{p_n}|
\end{align}
and consider iterative refinement for an approximation of $X=(x^{(p_1)},x^{(p_2)},\dots,x^{(p_k)})$.

We apply the proposed method to the shifted squared matrix $B=A^2-\alpha I$.
The eigenvalues of $B$ are $\lambda_i^2-\alpha$, and the corresponding eigenvectors are $x^{(i)}$.
If $\alpha$ is chosen such that
\begin{align}
    \max_{k+1\leq j\leq n}|\lambda_{p_j}^2-\alpha| < \min_{1\leq i\leq k}|\lambda_{p_i}^2-\alpha|
\end{align}
then the proposed method applied to $B$ satisfies the convergence condition.

A practical choice of $\alpha$ is to first compute an approximation $\widehat \lambda_{p_{k+1}}$ of $\lambda_{p_{k+1}}$ and set
\begin{align}
    (\|A\|^2+|\widehat \lambda_{p_{k+1}}|)/2\leq (\|A\|_{\infty}^2+|\widehat \lambda_{p_{k+1}}|)/2=\alpha.
\end{align}

Note that it is unnecessary to form $B=A^2-\alpha I$ explicitly.
In the proposed method, it suffices to compute the product $A^2\widehat X=A(A\widehat X)$, which requires no additional storage for $A^2$.
\section{Conclusion}\label{sec:conclusion}
We proposed an iterative refinement method that improves the accuracy of a selected subset of eigenvectors of a real symmetric matrix while using only $\mathcal{O}(nk)$ storage.
By representing the underlying orthogonal matrix in compact WY form, the refinement step can be implemented without forming dense $n\times n$ factors explicitly, and its dominant operations are matrix--matrix multiplications together with small $k\times k$ updates.
We provided an error-block analysis and established linear convergence under an eigenvalue separation condition, clarifying how convergence degrades as the target eigenvalues become more tightly clustered.
Numerical experiments on dense test problems and sparse matrices from the SuiteSparse Matrix Collection illustrated the attainable accuracy and the problem-dependent convergence behavior, and demonstrated that Rayleigh--Ritz preprocessing can restore convergence in a representative clustered case.
We also presented a conservative sufficient condition that helps delineate difficult instances (e.g., poor initial approximations).
In addition, we outlined practical variants (via shifting) that extend the separation-based analysis to other extremal parts of the spectrum.

Several directions remain for future work.
These include a performance-oriented implementation that exploits optimized level-3 BLAS and accelerator kernels, a systematic study of mixed-precision variants, and refined finite-precision analysis and stopping criteria that better capture the behavior in tightly clustered regimes.

\section*{Acknowledgments}

This work was supported by JSPS KAKENHI Grant Numbers 23K28100 and 25H00449.

\bibliographystyle{siamplain}
\bibliography{references}

\end{document}